\documentclass[12pt]{article}
\usepackage{geometry}
\usepackage{bbold}
\usepackage{graphicx} 
\usepackage{color}
\usepackage{float}
\usepackage{amsmath,amssymb,bm}
\usepackage{amsthm}

\newtheorem{propos}{Proposition}
\newtheorem{corol}{Corollary}
\newtheorem{definition}{Definition}
\newtheorem{exa}{Example}
\newtheorem{rem}{Remark}

\def\be{\begin{equation}}
\def\ee{\end{equation}}
\def\bea{\begin{eqnarray}}
\def\eea{\end{eqnarray}}

\newcommand*{\vect}[1]{\bm{#1}}

\newcommand{\f}{\vect{f}}

\newcommand{\g}{\vect{g}}

\newcommand{\x}{\vect{x}}

\def\bv{\mathbf{v}}
\def\x{{\bf x}}
\def\X{{\bf X}}
\def\Y{{\bf Y}}

\def\f{{\bf f}}

\def\bv{{\bf v}}
\def\bw{{\bf w}}
\def\A{{\bf A}}
\def\B{{\bf B}}

\def\U{{\bf U}}
\def\V{{\bf V}}
\def\W{{\bf W}}

\definecolor{orange}{rgb}{1,0.5,0}

\begin{document}

\title{Koopman Operator Spectrum for Random Dynamical Systems}
\author{Nelida \v{C}rnjari\'{c}-\v{Z}ic\footnote{Faculty of Engineering, University of Rijeka, Croatia, nelida@riteh.hr}, 
	Senka Ma\'{c}e\v{s}i\'{c}\footnote{Faculty of Engineering, University of Rijeka, Croatia, senka.macesic@riteh.hr}, 
	and 
	Igor Mezi\'{c}\footnote{Faculty of Mechanical Engineering and Mathematics, University of California, Santa Barbara, CA, 93106, USA mezic@engr.ucsb.edu}}

\maketitle

\begin{abstract}
	{
		In this paper we consider the Koopman operator associated with the discrete and the continuous time random dynamical system (RDS). 
		We provide  results that characterize the spectrum and the eigenfunctions of the stochastic Koopman operator associated with different types of linear RDS. Then we consider the RDS for which the associated Koopman operator family is a semigroup, especially those for which the generator can be determined. 
		We define a stochastic Hankel DMD (sHankel-DMD) algorithm for numerical approximations of the spectral objects (eigenvalues, eigenfunctions) of the stochastic Koopman operator and prove its convergence. We apply the methodology to a variety of examples, revealing objects in spectral expansions of the stochastic Koopman operator and enabling model reduction.}
	
	{\it Keywords:} stochastic Koopman operator, random dynamical systems, stochastic differential equations, dynamic mode decomposition
	
	{\it Mathematics Subject Classification: }  37H10, 47B33, 37M99, 65P99

\end{abstract}

\section{Introduction}
Prediction and control of the evolution of large complex dynamical systems is a modern-day science and engineering challenge. Some dynamical systems can be modeled well enough by using the standard mathematical tools, such as differential or integral calculus. In these cases the simplifications of the system are typically introduced by neglecting some of the phenomena with a small impact on the behavior of the system. However, there are many dynamical systems for which the mathematical model is too complex or even does not exist, but  data can be obtained by monitoring some observables of the  system. In this context, data-driven analysis methods need to be developed, including techniques to extract simpler, representative components of the process that can be used for modeling and prediction.

One approach for decomposing the complex systems into simpler structures is via the spectral decomposition of the associated Koopman operator. Koopman operator was introduced in \cite{Koopman:1931} in the measure-preserving setting, as a composition operator acting on the Hilbert space of square-integrable functions. The increased interest in the spectral operator-theoretic approach to dynamical systems in last decade starts with the works \cite{mezicandbanaszuk:2004} and \cite{mezic:2005}
(see also the earlier \cite{MezicandBanaszuk:2000}), where the problem of decomposing the evolution of an ergodic dissipative dynamical system from the perspective of operator theory was studied, and data-driven methods for computation of eigenvalues and eigenfunctions were provided based on rigorous harmonic analysis methods. The application of the theory to complex systems are numerous, for example in  fluid flows \cite{bagheri:2013,sharma2016}, infectious disease dynamics \cite{proctor}, power systems \cite{susukiandmezic:2012,susukiandmezic:2016} etc. The advantage of using the Koopman operator framework lies in the fact that it can be applied within the data-driven environment, even if the underlying mathematical model is not known. In addition to the analysis of the system properties the Koopman  decomposition of complex dynamical structures can provide approximations to the evolution of a possibly high dimensional system in lower dimensions, enabling  model reduction.  
An overview of the spectral properties of the Koopman operator and its applications prior to 2012 is given in \cite{Mezi2013,BudisicMezic_ApplKoop}.  

{ Of particular importance in Koopman operator theory in the deterministic case is the interplay between the geometry of the geometry of the underlying dynamics and  the level sets of the eigenfunctions of the Koopman operator. Namely, invariant sets can be determined as level sets of Koopman eigenfunctions at eigenvalue 
	$0$ (continuous time) or $1$ (discrete time), \cite{Mezic:1994,MezicandWiggins:1999}, periodic sets can be determined as level sets of Koopman operator eigenfunctions at the imaginary axis (continuous time) or the unit circle (discrete time), the concept of isostables - level sets of eigenfunctions associated with eigenvalues with negative real part (continuous time) or inside the unit circle (discrete time) generalize the geometric concept of Fenichel fibers \cite{Mauroyetal:2013}. Also, zero level sets of certain families 
	of eigenfunctions determine stable, unstable, and center manifolds to invariant sets \cite{Mezic:2015,Mezic:2017}. In this paper we enable generalization of such concepts by developing the theory of stochastic Koopman operators associated with RDS.} 

A variety of methods for determining the numerical approximation of the Koopman spectral decomposition, known under the name Koopman mode decomposition (KMD) have been developed.
KMD consists of the triple of Koopman eigenvalues, eigenfunctions and modes, which are the building blocks for the evolution of observables under the dynamics of a  dynamical system. A general method for computing the Koopman modes, based on the rigorous theoretical results for the generalized Laplace transform, is known under the name Generalized Laplace Analysis (GLA) \cite{Mezi2013,BudisicMezic_ApplKoop}, whose first version was advanced in \cite{MezicandBanaszuk:2000}. Another method is the Dynamic Mode Decomposition (DMD) method.
DMD was firstly introduced in \cite{schmid:2008} for the study of the fluid flows,   without any reference to the Koopman operator. The connection between the KMD and DMD was first pointed out in \cite{rowleyetal:2009}. Like GLA,  DMD  is a data-driven technique. Due to the fact it could be numerically implemented  relatively easily, since it relies on standard linear algebra concepts, this method become extremely popular in the data-driven systems research community. 
Starting with  \cite{schmid:2010} many versions of the DMD algorithm have been introduced to efficiently compute the spectral objects of the Koopman operator under various assumptions on the data. In \cite{Tu2014jcd} the exact DMD algorithm was introduced, while in \cite{williamsEDMD} the extension of the DMD algorithm, under the name Extended Dynamic Mode Decomposition (EDMD) was proposed. These two methods rely on the more standard approaches to representing a linear operator with respect to a specific basis (finite section method), rather than a sampling method - as in the companion-matrix based DMD (see \cite{Drmac2017}, section 2.1).  Susuki and Mezi\'{c} introduced in \cite{susukimezic} a further extension of the DMD algorithm, which combines the Prony method with the DMD algorithm, so that the Hankel matrix is used instead of the companion matrix to compute the Koopman spectrum on the single observable as well as on the vector of observable functions. In \cite{ArbabiMezic2016}, Arbabi and Mezi\'{c} referred to this algorithm as to the Hankel DMD and proved that under certain assumptions the obtained eigenvalues and eigenfunctions converge to the exact eigenvalues and eigenfunctions of the Koopman operator. The assumptions were removed, and further connections between the spectra of EDMD matrices and eigenvalues of the Koopman operator were proven in \cite{KordaandMezic:2017}.  

The majority of works analyzing or practically using the spectral properties of the Koopman operator assume that the dynamical system under consideration is autonomous. Similarly, the proposed numerical algorithms for evaluating KMD were almost exclusively applied to autonomous dynamical systems. The generalization of the Koopman operator framework to nonautonomous system was introduced in \cite{suranamezic}, where the definitions of the nonautonomous Koopman eigenvalues, eigenfunction and modes, as building blocks of the dynamical system, are given. KMD for nonautonomous systems was studied in \cite{MacesicZicMezic}, where the possibility of using the Arnoldi-like DMD algorithms for evaluating the Koopman eigenvalues and eigenfunctions in the nonautonomous case was carefully explored and the appropriate extensions of the algorithm were proposed. 

Another possible generalization of the Koopman operator framework is its extension to random dynamical systems (RDS) and stochastic systems. {There is a long history of analyzing the dynamics of the stochastic Markov processes through the study of the spectral properties of the associated transfer operators \cite{Dynkin,Yosida,LasotaMackey,ArnoldSDE}. In \cite{williamsEDMD}, it was realized that if the data provided to the EDMD algorithm are generated by a Markov stochastic process instead of a deterministic dynamical system, the algorithm approximates the eigenfunctions of the Kolmogorov backward equation. Klus and  coworkers used the Ulam's method and EDMD algorithm in the series of papers \cite{KlusKoltaiSchutte,KlusSchutte}  to approximate the spectral objects of the Koopman operator and its adjoint Perron-Frobenius operator. There is a variety of  numerical techniques used to approximate different transfer operators associated with the stochastic systems, as well as their spectral objects. A review can be found in \cite{KlusNuskeKoltaiSchutte}.}

{ In this paper we consider the stochastic Koopman operator associated with the discrete and  continuous-time RDS, using the definition and classification of RDS from \cite{Arnold}.}  {The Koopman operator for discrete RDS was first introduced in  \cite{MezicandBanaszuk:2000}, where the action of the stochastic Koopman operator on an observable was defined by taking an expectation of values of the observable at the next time step. The eigenspace at $1$ was shown to relate to invariant sets of stochastic dynamics,  followed by \cite{mezicandbanaszuk:2004,mezic:2005} in which analysis of the geometry of eigenspaces associated with any additional eigenvalues on the unit circle was pursued and shown to relate to periodic sets of the underlying RDS.} Here we extend the definition of the stochastic Koopman operator to the continuous-time RDS and explore some properties of the associated stochastic Koopman operator family, its generators,  eigenvalues and eigenfunctions. 

The characterization of the dynamics of such systems by using the eigenvalues and the eigenfunctions of the related generators was studied in recent papers \cite{ShnitzerTalmonSlotine}, \cite{TantetChekrounDijkstra} and \cite{Giannakis}. Giannakis in \cite{Giannakis} develops a framework for the KMD based on the representation of the Koopman operator in a smooth orthonormal bases determined from the time-ordered noisy data through the diffusion map algorithm. Using this representation, the Koopman eigenfunctions are approximated as the eigenfunctions of the related advection-diffusion operator. A similar approach by using the manifold learning technique via diffusion maps was used in \cite{ShnitzerTalmonSlotine}  to capture the inherent coordinates for building an intrinsic representation of the dynamics generated by the Langevin stochastic differential equation. The linear operator is then used to describe the evolution of the constructed coordinates on the state space of the dynamical system. The obtained coordinates are approximations of the eigenfunctions of the stochastic Koopman generator, so that the described approach is closely connected with the Koopman operator techniques for building the representation of the system dynamics. 

In order to numerically approximate the spectral objects of the stochastic Koopman operator, we explore the possibility of application of DMD algorithms that were originally developed for evaluating KMD in deterministic systems. As already mentioned, algorithms that are typically used to extract relevant spectral information are, for example, the Schmid DMD algorithm \cite{schmid:2010,schmid:2011}, the exact DMD algorithm \cite{Tu2014jcd}, and the Arnoldi-like algorithm \cite{rowleyetal:2009}. The application of DMD algorithm to noisy data is studied in \cite{hemati2015biasing} and \cite{TakeishiKawaharaYairi}. In order to remove the bias errors produced by using the standard DMD algorithms on data with the observation noise that can arise, for example, as a consequence of imprecise measurements, Hemati et al. developed in \cite{hemati2015biasing}  the total least squares DMD. Takeishi et al. \cite{TakeishiKawaharaYairi} considered the  numerical approximations of spectral objects of the stochastic Koopman operator for the RDS with observation noise by using the DMD algorithm, and proved its convergence to the stochastic Koopman operator under certain assumptions, following the work on deterministic systems in \cite{ArbabiMezic2016}. Due to the systematic error produced by the standard DMD algorithm they developed the version of the algorithm that also takes into account the observation noise and refer to it as the subspace DMD algorithm. Here we provide the convergence proof for the sHankel-DMD algorithm, an extension of the Hankel DMD algorithm.

There exist some numerical issues that can arise in DMD algorithms applied to deterministic systems, leading to poor approximations of eigenvalues and eigenvectors. In order to overcome these difficulties, Drma\v{c} et al. proposed  in \cite{Drmac2017} a data-driven algorithm for computing DMD, called DMD refined Rayleigh Ritz   (DMD RRR) algorithm, which enables selection of Ritz pairs based on the data-driven computation of the residual, and substantially improves the quality of the retained spectral objects. We use it in the current context of RDS. In most considered RDS examples, we determine first the Koopman eigenvalues and eigenfunctions of the related deterministic dynamical system and then explore the algorithm behavior on the considered RDS. Despite the fact that in some cases we noticed high sensitivity of the numerical algorithm to the noise introduced into the system, in most cases very satisfactory approximations are obtained.  

The paper is organized as follows. In Section \ref{sec:SKO}, the definition of the stochastic Koopman operator family for the discrete and the continuous-time RDS is given. { The classes of RDS that we consider are: the discrete time RDS, the continuous time RDS generated by the random differential equations (RDE) and the continuous time RDS generated by the stochastic differential equations (SDE). In accordance with this classification of RDS, we provide some results that characterize the spectral objects of the Koopman operator associated with the linear RDS.  In Section \ref{sec:semigroup} we limit our considerations to the RDS for which the associated stochastic Koopman operator family is a semigroup. These are the RDS which can be identified with the time homogeneous Markov stochastic processes. We also provide a characterization of the generators of the Koopman operator family.  
	In Section \ref{sec:NumSKO}, a brief description of the DMD RRR algorithm is provided and different approaches for its application to the RDS are described. We define the stochastic Hankel DMD (sHankel-DMD) algorithm for approximating the eigenvalues and eigenfunctions of the stochastic Koopman operator and prove, under the assumption of ergodicity, its convergence to the true eigenvalues and eigenfunctions of the stochastic Koopman operator.} Finally, the computations of Koopman eigenvalues and eigenfunctions by using the described DMD algorithms are illustrated on variety of numerical examples in Section \ref{sec:NumExamples}.

\section{Stochastic Koopman operator and Linear RDS}\label{sec:SKO}
First, we give the definition and a brief description of RDS, which follows terminology and results in \cite{Arnold,ArnoldCrauel}.

Let $(\Omega,\mathcal{F},P)$ be a probability space and $\mathbb{T}$ a semigroup (we can think of it as time). Suppose that $\theta:=(\theta(t))_{t \in \mathbb{T}}$ is a group or semigroup of  measurable transformations of $(\Omega,\mathcal{F},P)$ preserving a measure $P$ (i.e.~$\theta(t) P = P$), such that a map $(t,\omega) \rightarrow \theta(t)\omega$ is measurable. The quadruple $(\Omega,\mathcal{F},P,(\theta(t))_{t \in \mathbb{T}})$ is a metric dynamical system.

A random dynamical system (RDS) on the measurable space $(M,\mathcal{B})$ over a dynamical system  $(\Omega,\mathcal{F},P,(\theta(t))_{t \in \mathbb{T}})$ is a measurable map
$\varphi : \mathbb{T} \times \Omega \times M \rightarrow M$ satisfying a cocycle property with respect to $\theta(\cdot)$, which means that the mappings $\varphi(t,\omega):= \varphi(t,\omega,\cdot) : M \rightarrow M$ form the cocycle over $\theta(\cdot)$, i.e.:
\begin{equation}\label{eq:cocycle}
\varphi(0,\omega)= id_M, \ \ \varphi(t+s,\omega) = \varphi(t,\theta(s)\omega) \circ \varphi(s,\omega), \ \ \text{for all  } s, t \in \mathbb{T}, \omega \in \Omega.
\end{equation}
We call $\theta$ a driving flow and $(\Omega,\mathcal{F},P,(\theta(t))_{t \in \mathbb{T}})$ a driving dynamical system. 

{If $\mathbb{T}$ is discrete (for example, $\mathbb{T}=\mathbb{Z}$  or  $\mathbb{T}=\mathbb{Z}^{+}$), we speak about the discrete-time RDS, if $\mathbb{T}$ is continuous (for example, $\mathbb{T}=\mathbb{R}$  or  $\mathbb{T}=\mathbb{R}^{+}$), we have the continous-time RDS.  }

In the next definition we introduce the Koopman operator family associated with the RDS.
\begin{definition}\label{def:KoopmanStochasticGeneral}
	The stochastic Koopman operator $\mathcal{K}^{t}$ associated with the RDS $\varphi$ is defined on a { space of functions (observables) $f : M \rightarrow \mathbb{C}$ for which the functional} 
	\begin{equation}\label{eq:KoopmanStochasticGeneral}
	\mathcal{K}^{t} f(\x)=\mathbb{E} [f (\varphi(t,\omega) \x )], \quad \x \in M
	\end{equation}
	exists.
	We refer to the family of operators $\left(\mathcal{K}^{t}\right)_{t \in \mathbb{T}}$ as to the stochastic Koopman operator family. 
\end{definition}
{
	For example, if we work with continuous  functions on a compact metric space, the functional above will exist. In the case when spectral expansions of the Koopman operator are required, the specification of the space might become more detailed (e.g. some Hilbert space will suffice in many examples, see Example \ref{nrot}).} { However, in order to keep the exposition simple here we mostly (except in some examples) do not specify  the domain of the operator
	as our analysis largely does not depend on it. }

\begin{definition}\label{def:KoopmanEigenvalues}
	The observables $\phi^{t} : M \rightarrow \mathbb{C}$ that satisfy equation 
	\begin{equation} \label{KoopmanStochasticContinEigen}
	\mathcal{K}^{t} \phi^{t}(\x)= \lambda^{S}(t) \phi^{t}(\x)
	\end{equation}
	we call the eigenfunctions of the stochastic Koopman operator, while the associated values $\lambda^{S}(t)$ we call the stochastic Koopman eigenvalues.
\end{definition}

{ The above definition can be viewed geometrically: the level sets of a stochastic Koopman operator eigenfunction are unique in the sense that the expectation of the value of the level set at time $t$ is precisely $\lambda^{S}(t) \phi^{t}(\x)$ i.e. they depend only on the level set that the dynamics started from at $t=0$. An even stronger statement is available for eigenvalues on the unit circle: in that case, the precise location of the state of the system after time $t$ is within a level set of the associated eigenfunction. For example, if $\phi(\x)$ is an eigenfunction associated with eigenvalue $0$, then its level sets are invariant sets for the underlying RDS (see the proof in \cite{mezicandbanaszuk:2004} for the discrete-time RDS).  }

{
	There are three main types of RDS \cite{Arnold} that are particularly important from the point of view of practical applications. These are: the discrete time RDS, the continuous time RDS generated by RDE and the continuous time RDS generated by SDE. In what follows we describe the general settings for each of these types and provide the results that characterize the eigenvalues and eigenfunctions of the stochastic Koopman operators associated with the linear RDS.}

\subsection{Stochastic Koopman operator for discrete time RDS}\label{subsec21:DiscreteRDS}
For a discrete time RDS we have $\mathbb{T} = \mathbb{Z}^{+} \cup \{0\}$. Let $(\Omega,\mathcal{F},P, (\theta(t))_{t \in \mathbb{T}})$ be a given metric dynamical system and let denote $\psi=\theta(1)$. {The discrete RDS $\varphi(n,\omega)$ over $\theta$ can be defined by the one step maps $T(\omega,\cdot) : M \rightarrow M$ 
	$$T(\omega,\cdot):=\varphi(1,\omega),$$
	since by applying the cocycle property one gets
	\begin{equation}\label{eq:DRDS_Cocycle}
	\varphi(n,\omega) = T(\psi^{n-1}(\omega),\cdot) \circ \cdots \circ T(\psi(\omega),\cdot) \circ T(\omega,\cdot), \ \ n \ge 1.
	\end{equation} 
	From the fact that maps $\psi^{i}=\theta(i)$ preserve the measure $P$, it follows that the maps $T(\psi^{i}(\omega),\cdot)$ are identically distributed, thus $(T(\psi^{i}(\omega),\cdot))_{i \in \mathbb{T}}$ is a stationary sequence of random maps on $M$ \cite[Section~2.1]{Arnold}. }
According to (\ref{eq:cocycle}), $\varphi(0,\omega) = id_M$.
The action of the discrete RDS $\varphi(n,\omega)$ on $\x \in M$ gives its value at $n$-th step and we denote it as
\begin{equation}\label{eq:DRDS_Tn}
T^{n}(\omega, \x) = \varphi(n,\omega) \x.
\end{equation}

In the proposition that follows we consider the eigenvalues and the eigenfunctions of the stochastic Koopman operator related to the discrete RDS induced by a linear map $T$.

\begin{propos}\label{LDRDSthm}
	Suppose that $\A : \Omega \rightarrow \mathbb{R}^{d \times d}$ is measurable and that the one step map $T:\Omega\times\mathbb{R}^{d}\rightarrow \mathbb{R}^{d}$ of a discrete RDS is defined by 
	\begin{equation}\label{eq:LS_Discrete}
	T(\omega,\x)=\A(\omega) \x.
	\end{equation}
	{
		Denote by $\mathbf{\Phi}(n,\omega)$ the linear RDS satisfying $T^{n}(\omega,\x) = \mathbf{\Phi}(n,\omega) \x$,
		i.e., 
		$$\mathbf{\Phi}(n,\omega) = \A(\psi^{n-1}(\omega)) \cdots \A(\psi(\omega))\A(\omega).$$
		Assume that $\hat{\mathbf{\Phi}}(n) = \mathbb{E} [\mathbf{\Phi}(n,\omega)]$ are diagonalizable, with simple eigenvalues $\hat{\lambda}_j(n)$} and left and right eigenvectors $ \hat{\bw}_j^{n},\ \hat{\bv}_j^{n}$,  $j=1,\ldots,d$.
	Then the eigenfunctions of the stochastic Koopman operator $\mathcal{K}^n$ are
	\begin{equation}\label{eq:LS_Discrete_ef}
	\phi_j^{n}(\x)=\langle \x,  \hat{\bw}_j^{n} \rangle,  \ \ j=1,\ldots,d,
	\end{equation}
	with the corresponding eigenvalues { ${\lambda }^{S}_j(n)=\hat{\lambda}_j(n)$.}
	
	Moreover, if matrices $\A(\omega)$, $\omega\in\Omega$ commute and are diagonalizable with the simple eigenvalues $\lambda_j(\omega)$ and corresponding left eigenvectors $\bw_j, \ j=1,\ldots,d$, then 
	$$\hat{\bw}_j^{n}=\bw_j \quad \textrm{   and    } \quad { {\lambda }^{S}_j(n)=\mathbb{E}\left[\prod_{i=1}^{n}\lambda_j(\psi^{i-1}(\omega))\right].}$$
	
	{
		Furthermore, $\hat{\bv}_{j}^{n}$, $j=1,\ldots,d$ are the Koopman modes of the full-state observable and the following expansion is valid
		\begin{equation}\label{eq:KMD_RDS}
		\mathcal{K}^{n}\x=\displaystyle{\sum_{j=1}^d {\lambda_{j}^{S}(n)} \langle \x, \hat{\bw}_j^{n} \rangle   \hat{\bv}_j^{n}}.
		\end{equation}	
	}
\end{propos}

\begin{proof}
	The action of the stochastic Koopman operator on functions defined by (\ref{eq:LS_Discrete_ef}) is equal to
	{
		\begin{align}
		\mathcal{K}^{n} \phi_j^{n}( \x)&=\mathbb{E}[\langle \mathbf{\Phi}(n,\omega)\x, \hat{\bw}_j^{n} \rangle] =\mathbb{E}[\langle \x, \mathbf{\Phi}(n,\omega)^* \hat{\bw}_j^{n} \rangle] =\langle \x, \mathbb{E} [\mathbf{\Phi}(n,\omega)^*]\hat{\bw}_j^{n} \rangle \nonumber \\
		& = \langle \x, \hat{\mathbf{\Phi}}(n)^{*} \hat{\bw}_j^{n} \rangle = \hat{\lambda }_j(n) \langle \x, \hat{\bw}_j(n) \rangle = \hat{\lambda }_j(n) \phi_j^{n}( \x) ,
		\end{align}
		where we have used that 
		$$\mathbb{E} [\mathbf{\Phi}(n,\omega)^*]=(\mathbb{E} [\mathbf{\Phi}(n,\omega)])^* = \hat{\mathbf{\Phi}}(n)^{*}.$$
	}
	
	In the case when matrices $\A(\omega), \omega \in \Omega$ commute and are diagonalizable, {they are simultaneously diagonalizable (see \cite[Theorem~1.3.12]{HornJohnson}),~i.e., there exists a single invertible matrix $\V$, so that $\A(\omega)= \V \boldsymbol{\Lambda}(\omega) \V^{-1}$, where $\boldsymbol{\Lambda}(\omega) = \mbox{diag}(\lambda_{1}(\omega),\ldots,\lambda_{d}(\omega))$. It is clear that the columns of the matrix $\V$ are the common right eigenvectors of the matrices $\A(\omega), \omega \in \Omega$, and that $\W^{*}=\V^{-1}$, where $\W$ is the matrix of the left eigenvectors. It is straightforward that 
		$$\mathbf{\Phi}(n,\omega)= \V \prod_{i=1}^{n}\mathbf{\Lambda}(\psi^{i-1}(\omega)) \W^{*}. $$
		Therefore
		$$\hat{\mathbf{\Phi}}(n)= \V \mathbb{E} \left[\prod_{i=1}^{n}\mathbf{\Lambda}(\psi^{i-1}(\omega)) \right] \W^{*},$$
		and we easily conclude that $\hat{\bw}_j=\bw_j$ and ${\lambda^{S}_j(n)}=\mathbb{E} \left[\prod_{i=1}^{n}{\lambda_j}(\psi^{i-1}(\omega)) \right]$.}
	
	{Equation (\ref{eq:KMD_RDS}) can be easily proved by using the decomposition of the state $\x$ in the base $ \hat{\bw}_j^{n},\ \hat{\bv}_j^{n}$,  $j=1,\ldots,d$, the linearity of the Koopman operator, and its action on the eigenfunctions given with (\ref{eq:LS_Discrete_ef}), i.e.,
		\begin{equation}
		\mathcal{K}^{n}\x=\mathcal{K}^{n} \displaystyle{\sum_{j=1}^d \langle \x, \hat{\bw}_j^{n} \rangle   \hat{\bv}_j^{n}} = \displaystyle{\sum_{j=1}^d \mathcal{K}^{n}  \langle \x, \hat{\bw}_j^{n} \rangle   \hat{\bv}_j^{n}} =  \displaystyle{\sum_{j=1}^d {\lambda_{j}^{S}(n)} \langle \x, \hat{\bw}_j^{n} \rangle   \hat{\bv}_j^{n}}.
		\end{equation}	}
\end{proof}

\begin{rem} We will use the term principal eigenfunctions for the eigenfunctions of the form $\phi_j(\x) = \langle \x, \bw_j \rangle $, $j=1,\ldots,d$ that appear in Proposition \ref{LDRDSthm} and in some other propositions in the rest of the paper. Also, the eigenvalues corresponding to them we call the principal eigenvalues.
	{ It follows from (\ref{eq:KMD_RDS}) that in the linear case, the action of the Koopman operators on the full state observable can be derived by using just the principal eigenfunctions, eigenvalues and modes.  }
\end{rem}

\vspace{5pt}

\subsection{Stochastic Koopman operator for RDS { generated} by the random differential equations}
Suppose that $\mathbb{T}=\mathbb{R}$ or $\mathbb{T} = \mathbb{R}^{+} \cup \{0\}$. Let $(\Omega,\mathcal{F},P, (\theta(t))_{t \in \mathbb{T}})$ be a metric dynamical system.
We consider a continuous-time RDS generated by the random differential equation (RDE) of the following form
\begin{equation}\label{eq:RDE}
\dot{\x}=F(\theta(t)\omega,\x),
\end{equation}
defined on the manifold $M$, where $\theta(t)\omega$ is associated with the random dynamics.
In this type of equations the randomness refers just to the random parameters, which do not depend on the state of the system. RDE (\ref{eq:RDE}) generates an RDS $\varphi$ over $\theta$, whose action is defined by
\begin{equation}\label{eq:RDSbyRDE}
\varphi(t,\omega) \x = \x + \int_{0}^{t} F(\theta(s)\omega,\varphi(s,\omega) \x ) ds.
\end{equation}
The properties of this RDS under different regularity properties of the function $F$ can be found in {\cite[Section~2.2]{Arnold}}.
A set of trajectories starting at $\x$ that are generated by (\ref{eq:RDE}) is given by
$\varphi({t, \omega}) \x$ and it defines the family of random variables. We say that this is a solution of the RDE with the initial condition $\varphi(0,\omega) \x=\x$. Since the solutions of the RDE are defined pathwise, for each fixed  $\omega$ the trajectory can be determined as a solution of deterministic ordinary differential equation, so that the RDE  $(\ref{eq:RDE})$ can be seen as a family of ordinary differential equations. 

In the next proposition we derive the eigenvalues and the eigenfunctions of the stochastic Koopman operators associated to the linear RDS.

\begin{propos}\label{thm2}
	If $\mathbf{A}:\Omega\rightarrow\mathbb{R}^{d \times d}$ and $\mathbf{A}\in L^{1}(\Omega,\mathcal{F},{P})$, then RDE
	\begin{equation}\label{eq:LRDE}
	\dot{\x}=\mathbf{A}(\theta(t)\omega)\x,
	\end{equation}
	generates a linear RDS $\mathbf{\Phi}$ satisfying
	\begin{equation}\label{eq:linRDES}
	\mathbf{\Phi}(t,\omega)=\mathbf{I} + \int_{0}^{t}\mathbf{A}(\theta({s})\omega)\mathbf{\Phi}(s,\omega)ds.
	\end{equation}
	Assume that $\hat{\mathbf{\Phi}}(t) = \mathbb{E} [\mathbf{\Phi}(t,\omega)]$ is diagonalizable, with simple eigenvalues $\hat{\mu}_j^{t}$ and left eigenvectors $ \hat{\bw}^{t}_j$, $j=1,\ldots,d$.
	Then
	\begin{equation}\label{eq:linRDES_ef}
	\phi^{t}_j(\x)=\langle \x,  \hat{\bw}^{t}_j \rangle,  \ \ j=1,\ldots,d,
	\end{equation}
	are the principal eigenfunctions of the stochastic Koopman operator $\mathcal{K}^{t}$ with corresponding principal eigenvalues $\lambda_j^{S}(t)=\hat{\mu}_j(t)$, $j=1,\ldots,d$. 
	\\
	Moreover, if matrices $\A(\omega)$ commute and are diagonalizable with the simple eigenvalues $\lambda_j(\omega)$ and corresponding left eigenvectors $\bw_j, \ j=1,\ldots,d$, then 
	$$\hat{\bw}^{t}_j=\bw_j \quad \textrm{   and    }  \quad {\lambda_j^{S}(t)}=\mathbb{E} \left[{\rm e}^{\int_{0}^{t}\lambda_j(\theta(s)\omega)ds}\right].$$
	{
		Furthermore, $\hat{\bv}_{j}^{t}$, $j=1,\ldots,d$ are the Koopman modes of the full-state observable and the following expansion is valid
		\begin{equation}\label{eq:KMD_RDS_cont}
		\mathcal{K}^{t}\x=\displaystyle{\sum_{j=1}^d {\lambda_{j}^{S}(t)} \langle \x, \hat{\bw}_j^{t} \rangle   \hat{\bv}_j^{t}}.
		\end{equation}	
	}
\end{propos}

\begin{proof}
	The first part of the proposition follows from \cite[Example 2.2.8]{Arnold}, 
	
	Furthermore, the action of the stochastic Koopman operator on functions defined by (\ref{eq:linRDES_ef}) is equal to
	\begin{align}
	\mathcal{K}^{t} \phi^{t}_j( \x) 
	&=\mathbb{E}[\langle \mathbf{\Phi}(t,\omega)\x, \hat{\bw}^{t}_j \rangle] 
	=\mathbb{E}[\langle \x, \mathbf{\Phi}^*(t,\omega) \hat{\bw}^{t}_j \rangle] 
	=\langle \x, \mathbb{E} [\mathbf{\Phi}^*(t,\omega)]\hat{\bw}^{t}_j \rangle \nonumber \\
	&= \langle \x, \hat{\mathbf{\Phi}}^{*}(t) \hat{\bw}^{t}_j \rangle 
	= \hat{\mu }^{t}_j \langle \x, \hat{\bw}^{t}_j \rangle 
	= \hat{\mu }^{t}_j \phi^{t}_j( \x)
	\end{align}
	With the same argument as in the proof of Proposition \ref{LDRDSthm}, we have $\A(\omega) = \V \boldsymbol{\Lambda}(\omega)\W^{*} $, where $\W$ and $\V$ are matrices of common left and right eigenvectors, and $\boldsymbol{\Lambda}(\omega) = \mbox{diag}(\lambda_{1}(\omega),\ldots,\lambda_{d}(\omega))$. It is straightforward that  
	$$\mathbf{\Phi}(t,\omega)= \mbox{e}^{\int_{0}^{t}\A(\theta(s)\omega)ds}= \V \mbox{e}^{\int_{0}^{t}\mathbf{\Lambda}(\theta(s)\omega)ds} \W^{*},$$
	and
	$$\hat{\mathbf{\Phi}}(t)= \V \mathbb{E} \left[\mbox{e}^{\int_{0}^{t}\mathbf{\Lambda}(\theta(s)\omega)ds} \right] \W^{*}.$$
	We easily conclude that $\hat{\bw}_j=\bw_j$ and ${\lambda^{S}_j(t)}=\mathbb{E} \left[\mbox{e}^{\int_{0}^{t}\lambda_j(\theta(s)\omega)ds}\right]$.
	{ The proof of (\ref{eq:KMD_RDS_cont}) is same as the proof of (\ref{eq:KMD_RDS}) in Proposition \ref{thm2}.}
\end{proof}

\begin{exa}[Linear scalar RDS]
	{Suppose that linear scalar RDE is given by
		\begin{equation}\label{eq:rls-1d}
		\dot{x} = a(\omega) x, 
		\end{equation} 
		where $a :\Omega \rightarrow \mathbb{R}$ is random variable with finite moments. Observe that $\theta(t) = id$, which means that the probability space does not change with time.} 
	If the moment generating function defined by $\displaystyle \mathcal{M}_{a}(t)=\mathbb{E}[\mbox{e}^{t a(\omega)}]$ is analytic for $|t| < R$, then for the initial condition $\varphi(0,\omega) x  = x$ and $t < R$ there exists a unique solution of (\ref{eq:rls-1d}), which can be expressed as \cite{Strand}
	\begin{equation}\label{eq:rls-sol}
	\varphi(t,\omega) x  = x \mbox{e}^{\int_{0}^{t}a(\omega) ds}.
	\end{equation} 
	
	The action of the stochastic Koopman operator on the full state observable function $\phi(x)=x$ is then 
	\begin{equation}\label{eq:rls-Koopman}
	\mathcal{K}^{t} x = \mathbb{E} \left[ \varphi(t,\omega)x \right]  = \mathbb{E} \left[ x \mbox{e}^{\int_{0}^{t} a(\omega)) ds} \right] = \mathbb{E} \left[ \mbox{e}^{a(\omega) t} \right] \, x.
	\end{equation}	
	Thus, $\phi(x) = x$ is the eigenfunction of the stochastic  Koopman operator and the corresponding eigenvalue satisfies 
	$${\lambda^{S}(t)} =  \mathbb{E} \left[ \mbox{e}^{a(\omega) t} \right] = \mathcal{M}_{a}(t).$$
\end{exa}

\subsection{Stochastic Koopman operator for RDS generated by the stochastic differential equations}\label{subs:SDENonaut}
Let $\mathbb{T}=\mathbb{R}^{+}$, $M = \mathbb{R}^{d}$ and $T \in \mathbb{T}$. Suppose that the stochastic process $X_t(\omega)$, $t \in [0,T]$, $\omega \in \Omega$ is obtained as a solution of the nonautonomous stochastic differential equation (SDE)
\begin{equation}\label{eq:SDE_Nonaut}
dX_t = G(t, X_t) dt + {\sigma}(t, X_t) dW_t,
\end{equation} 
where $G : [0,T] \times \mathbb{R}^{d} \rightarrow \mathbb{R}^{d}$ and ${\sigma} : [0,T] \times \mathbb{R}^{d} \rightarrow \mathbb{R}^{d \times r}$ are $L^{2}$ measurable. 
Here $W_t=(W_t^{1},\ldots,W_t^{r})$ denotes the $r$-dimensional Wiener process with independent components and standard properties, i.e., $\mathbb{E}(W_t^{i}) = 0$, $i=1,\ldots,r$,  $\mathbb{E}(W_t^{i} W_{s}^{j}) = \min\{t,s\} \delta_{ij}$, $i,j = 1,\ldots,r$ ($\delta_{ij}$ is the Kronecker delta symbol).  
The solution $X_t(\omega)$ of (\ref{eq:SDE}) with the initial condition $X_{t_0}(\omega)$ is formally defined in terms of It\^{o} integral as {(see \cite[Chapter 6]{ArnoldSDE})}
\begin{equation}\label{eq:SDE_xw}
X_t(\omega) = X_{t_0}(\omega) + \int_{t_0}^{t} G(s, X_s(\omega)) ds + \int_{t_0}^{t} {\sigma}(s, X_s(\omega)) dW_s.
\end{equation}
The probability space on which the process is considered can be identified with $\Omega = \mathcal{C}_0(\mathbb{R}^{+}, \mathbb{R}^{r})$ (space of continuous functions satisfying $\omega(t_0)=0$). Then, $\omega \in \Omega$ is identified with the canonical realization of the Wiener process such that $\omega(t)=W_t(\omega)$. If $\mathcal{F}$ is the Borel $\sigma$-algebra, $P$ the measure generated by the Wienner process, and $\theta(t)$ defined by the "Wiener shifts"
\begin{equation}\label{eq:WienerShift}
\theta(t)\omega(\cdot) = \omega(t+\cdot) - \omega(t),
\end{equation}
$(\Omega,\mathcal{F},P,(\theta(t))_{t \in \mathbb{T}})$ becomes a metric dynamical system {(see \cite[Appendix~A]{Arnold}).} 
It is a driving dynamical system for the two parameter family of RDS $\varphi({t, t_0, \omega})$ that is given by
$$\varphi({t, t_0, \omega}) \x =X_t(\omega),$$
where $X_t(\omega)$ is given with (\ref{eq:SDE_xw}) for the initial condition $X_{t_0}(\omega)=\x$.
{ Under certain regularity and boundedness properties of the functions $G$ and $\sigma$. The basic existence and uniqueness results for the SDE of the form (\ref{eq:SDE_Nonaut}) can also be found in \cite[Section~6.3]{ArnoldSDE}, \cite[Section 3.3]{Pavliotis}.}

The Koopman operator family $\mathcal{K}^{t,t_0}$ related to this RDS is defined by \begin{equation}\label{eq:KoopmanStochasticGeneral_NASDE}
\mathcal{K}^{t,t_0} f(\x)=\mathbb{E} [f (\varphi(t, t_0, \omega) \x )].
\end{equation}
In this more general setting with the two-parameter family of Koopman operators  (\ref{eq:KoopmanStochasticGeneral_NASDE}), the eigenfunctions $\phi^{t,t_0} : M \rightarrow \mathbb{C}$ and eigenvalues $\lambda^{S}(t,t_0)$ of the Koopman operator $\mathcal{K}^{t,t_0}$  defined on a finite-time interval satisfy
\begin{equation} \label{KoopmanStochasticContinEigen_NASDE}
\mathcal{K}^{t,t_0} \phi^{t,t_0}(\x)= \lambda^{S}(t,t_0) \phi^{t,t_0}(\x). 
\end{equation}
\\
The following two propositions treat two classes of linear SDE. In the first one the random part of equations models the additive noise and in the second one it models the multiplicative noise. 

\begin{propos}\label{thm3}
	Let the linear SDE with additive noise be defined by
	\begin{equation}\label{eq:LSDE}
	dX_t =\mathbf{A}(t) X_t dt + \sum_{i=1}^{r} b^{i}(t) dW_t^{i},
	\end{equation}
	where $\mathbf{A}(t)$ is $d \times d$ matrix of functions and $b^{i}(t)$, $i=1,\ldots,m$ are $d$-dimensional vector functions.
	Assume that the fundamental matrix $\mathbf{\Phi}(t,t_0)$ satisfying the matrix  differential equation 
	\begin{equation}\label{eq:linSDES}
	\dot{\mathbf{\Phi}}= \mathbf{A}(t) \mathbf{\Phi},\quad \mathbf{\Phi}(t_0)={\bf I}
	\end{equation}
	is diagonalizable, with simple eigenvalues $\hat{\mu}_j^{t,t_0}$ and left eigenvectors $ \hat{\bw}^{t,t_0}_j$, $j=1,\ldots,d$.
	Then 
	\begin{equation}\label{eq:linSDES_ef}
	\phi^{t,t_0}_j(\x)=\langle \x, \hat{\bw}^{t,t_0}_j \rangle,  \ \ j=1,\ldots,d,
	\end{equation}
	are the eigenfunctions of the stochastic Koopman operator $\mathcal{K}^{t,t_0}$, with corresponding eigenvalues $$\lambda_j^{S}(t,t_0)=\hat{\mu}_j(t,t_0).$$ 
	If matrices $\A(t)$ commute and are diagonalizable with the simple eigenvalues $\lambda_j(t)$ and corresponding left eigenvectors $\bw_j, \ j=1,\ldots,d$, then 
	\begin{equation}\label{eq:LSDE_Aut_eigens}
	\hat{\bw}^{t,t_0}_j=\bw_j \quad \textrm{   and    } \quad \lambda_j^{S}(t,t_0)= \mbox{e}^{\int_{t_0}^{t}\lambda_j(s) ds}.
	\end{equation}
\end{propos}
\begin{proof}
	Since the solution of (\ref{eq:LSDE}) with the initial condition $X_{t_0}(\omega)=\x$ is given by, {(see \cite[Section~8.2]{ArnoldSDE})}
	\begin{equation} \label{eq:LSDE_Xt}
	X_t(\omega)  = \mathbf{\Phi}(t,t_0) \left(\x + \sum_{i=1}^{r} \int_{t_0}^{t} \mathbf{\Phi}^{-1}(s,t_0) b^{i}(s) dW_s^{i} \right),
	\end{equation}
	we have
	\begin{align}
	& \mathcal{K}^{t,t_0} \phi^{t,t_0}_j(\x)  = \mathbb{E} \left[ \phi^{t,t_0}_j (X_t(\omega)) \right] \nonumber \\
	& = \mathbb{E} \left[ \big\langle \mathbf{\Phi}(t,t_0) \x,  \hat{\bw}^{t,t_0}_j \big\rangle +  \big\langle \sum_{i=1}^{r} \int_{t_0}^{t} \mathbf{\Phi}(t,t_0) \mathbf{\Phi}^{-1}(s,t_0) b^{i}(s) dW_s^{i}, \hat{\bw}^{t,t_0}_j \big\rangle \right] \nonumber \\
	& = \big\langle \mathbf{\Phi}(t,t_0)\x,  \hat{\bw}^{t,t_0}_j \big\rangle + \sum_{i=1}^{r} \big\langle \mathbb{E} \left[   \int_{t_0}^{t} \mathbf{\Phi}(t,t_0) \mathbf{\Phi}^{-1}(s,t_0) b^{i}(s) dW_s^{i} \right], \hat{\bw}^{t,t_0}_j  \big\rangle
	\nonumber \\
	& = \hat{\mu}_j^{t,t_0} \langle \x,  \hat{\bw}^{t,t_0}_j \rangle= \hat{\mu}_j^{t,t_0} \phi^{S,t,t_0}_j(\x), 
	\end{align}
	where we used that { $\displaystyle \mathbb{E} \left[   \int_{t_0}^{t} {\bf F}(s) dW_s \right] = 0$ (see \cite[Theorem~4.4.14]{ArnoldSDE}) applied to ${\bf F}(s) = \mathbf{\Phi}(t,t_0) \mathbf{\Phi}^{-1}(s,t_0) b^{i}(s)$}. With this we proved the first statement. 
	
	Since in the commutative case the fundamental matrix can be expressed in the form 
	$\mathbf{\Phi}(t,t_0) = \mbox{e}^{\int_{t_0}^{t} \A(s) ds}$, its eigenvectors coincide with the eigenvectors of the matrix $\A(t)$ and eigenvalues are given by (\ref{eq:LSDE_Aut_eigens}). 
\end{proof}

\begin{propos}\label{thm4}
	Let the linear SDE with multiplicative noise be defined by
	\begin{equation}\label{eq:LSDE_2}
	dX_t =\mathbf{A}(t) X_t  dt + \sum_{i=1}^{r} \B^{i}(t) X_t \,dW_t^{i},
	\end{equation}
	where $\mathbf{A}(t)$, $\B^{i}(t)$, $i=1,\ldots,r$ are $d \times d$ matrices of functions. 
	Denote with $\mathbf{\Phi}(t,t_0)$ the fundamental matrix satisfying the matrix SDE
	\begin{equation}\label{eq:linSDES_2}
	d \mathbf{\Phi} = \mathbf{A} \mathbf{\Phi} \, dt  + \sum_{i=1}^{r}\B^{i}(t) \mathbf{\Phi} \, dW_t^{i},\quad \mathbf{\Phi}(t_0)={\bf I}
	\end{equation}
	and assume that $\hat{\mathbf{\Phi}}(t,t_0) = \mathbb{E} \left[\mathbf{\Phi}(t,t_0) \right]$ is diagonalizable, with simple eigenvalues $\hat{\mu}_j^{t,t_0}$ and left eigenvectors 
	$\hat{\bw}^{t,t_0}_j$, $j=1,\ldots,d$.
	Then
	\begin{equation}\label{eq:linSDES_ef1}
	\phi^{t,t_0}_j(\x)=\langle \x,  \hat{\bw}^{t,t_0}_j \rangle,  \ \ j=1,\ldots,d,
	\end{equation}
	are the eigenfunctions of the stochastic Koopman operator $\mathcal{K}^{t,t_0}$, with corresponding eigenvalues 
	$$\lambda_j^{S}(t,t_0)=\hat{\mu}_j(t,t_0).$$ 
	If the matrices $\A(t), \B^{i}(t)$, $i=1,\ldots,r$ commute, i.e.~if $\A(t) \A(s)= \A(s) \A(t)$, $\A(t) \B^{i}(s) = \B^{i}(s)  \A(t), \B^{i}(t) \B^{j}(s) = \B^{j}(s) \B^{i}(t)$ for $i,j=1,\ldots,r$ and all $s, t$, and if the matrices $\A(t)$ are diagonalizable with the simple eigenvalues  $\lambda_j(t)$ and corresponding left eigenvectors $\bw_j, \ j=1,\ldots,d$, then 
	\begin{equation}\label{eq:LSDE_Aut_eigens2}
	\hat{\bw}^{t,t_0}_j=\bw_j \quad \textrm{   and    } \quad \lambda_j^{S}(t,t_0)= \mbox{e}^{\int_{t_0}^{t}\lambda_j(s) ds}.
	\end{equation}
\end{propos}
\begin{proof}
	For the fundamental matrix $\mathbf{\Phi}(t,t_0)$ and the initial condition $X_{t_0}(\omega)=\x$,  the solution of (\ref{eq:LSDE_2}) is equal to, { (see \cite[Section~8.5]{ArnoldSDE})}
	\begin{equation} \label{eq:LSDE_Xt_2}
	X_t(\omega) = \mathbf{\Phi}(t,t_0) \x,
	\end{equation}
	thus
	\begin{align}
	& \mathcal{K}^{t,t_0} \phi^{t,t_0}_j(\x) = \mathbb{E} 
	\left[ \phi^{t,t_0}_j (X_t(\omega)) \right] = \mathbb{E} 
	\left[ \big\langle \mathbf{\Phi}(t,t_0) \x,  \hat{\bw}^{t,t_0}_j \big\rangle \right] \nonumber \\
	& = \big\langle \hat{\mathbf{\Phi}}(t,t_0)\x,  \hat{\bw}^{t,t_0}_j \big\rangle = {\mu}_j^{t,t_0} \langle \x,  \hat{\bw}^{t,t_0}_j \rangle= e^{\hat{\lambda}_j(t,t_0)} \phi^{t,t_0}_j(\x), 
	\end{align}
	For the case with commutative matrices $\A(t)$ and $\B^{i}(t)$, $i=1,\ldots,r$, the fundamental matrix $\mathbf{\Phi}(t,t_0)$ can be expressed in an explicit form as, {(see \cite[Sections~8.4, 8.5]{ArnoldSDE})}
	\begin{equation}\label{eq:linSDE_Phi_2}
	\mathbf{\Phi}(t,t_0)= \mbox{e}^{{\int_{t_0}^{t}{\big(\A(s) - \frac{1}{2} \sum_{i=1}^{r} \B^{i}(s) \B^{i}(s)^{T} \big) ds} + \int_{t_0}^{t}{ \sum_{i=1}^{r} \B^{i}(s) dW_s^{i}}} }.
	\end{equation}
	Since
	\begin{align}\label{eq:linSDE_Phi_2Exp}
	\hat{\mathbf{\Phi}}(t,t_0) &= \mathbb{E} \left[{\mathbf{\Phi}}(t,t_0)\right] 
	=\mbox{e}^{ {\int_{t_0}^{t}{\big(\A(s) - \frac{1}{2} \sum_{i=1}^{r} \B^{i}(s) \B^{i}(s)^{T} \big) ds}} } \  \mathbb{E} \bigg[ \mbox{e}^{ \int_{t_0}^{t}
		{\sum_{i=1}^{r} \B^{i}(s) dW_s^{i}} }\bigg] \nonumber \\
	& = \mbox{e}^{{\int_{t_0}^{t}{\big(\A(s) - \frac{1}{2} \sum_{i=1}^{r} \B^{i}(s) \B^{i}(s)^{T} \big) ds}} }  \ \mbox{e}^{\int_{t_0}^{t}
		{\frac{1}{2} \sum_{i=1}^{r} \B^{i}(s) \B^{i}(s)^{T} ds} }
	= \mbox{e}^{{\int_{t_0}^{t}{\A(s) ds}} },
	\end{align}
	the eigenvectors of $\hat{\mathbf{\Phi}}(t,t_0)$ coincide with the eigenvectors of the matrix $\A(t)$ and eigenvalues are given by (\ref{eq:LSDE_Aut_eigens2}), which proves the statement. Here we used the fact that 
	$\mathbb{E} \left[\mbox{e}^{\int_{t_0}^{t} \B(s) dW_s}\right] =\mbox{e}^{\frac{1}{2} \int_{t_0}^{t} \B(s) \B(s)^{T} ds}$ { (see \cite[Theorem~8.4.5]{ArnoldSDE}).}
\end{proof}
\vspace{10pt}

\section{Semigroups of Koopman operators and their generators}\label{sec:semigroup}
In this section we are interested in RDS for which the associated stochastic Koopman operator family satisfy the semigroup property.  
{
	In what follows, we limit our considerations to the situation where $M \subseteq \mathbb{R}^{d}$ (or more generally $M$ could be a Polish space) and $\mathcal{B}$ is the Borel $\sigma$-algebra of Borel sets in $M$, so that $(M;\mathcal{B})$ is a standard space. For a fixed $\x \in M$, $(\varphi(t,\omega)\x)_{t\in \mathbb{T},\omega \in \Omega}$ is the $M$-valued stochastic process, which implies that an RDS defines a family of random processes with values in $M$. For each $\x \in M$, the corresponding probability measure $\mathbb{P}_\x$ on the canonical space $\bar{\Omega} \subseteq M^{\mathbb{T}}$ is induced by the probability measure ${P}$ associated with the driving dynamical system and a process $(\varphi(t,\omega)\x)_{t\in \mathbb{T},\omega \in \Omega}$ (see \cite{Dynkin}, \cite[Appendix A.2]{Arnold}).

	Strong results for the properties of the stochastic Koopman operator family can be obtained in more specific settings such as linear setting we analyzed before and Markov setting. Therefore, we limit here to the particular types of RDS and consider only those that are Markovian \cite{Crauel_MarkovRDS}, or even more precisely, to those RDS for which the family of processes $\{(\varphi(t,\omega) \x), \x \in M\}$ is a time-homogeneous Markov family. Denote by $\mathcal{F}_{t}^{\vec{x}}$ the $\sigma$-algebra generated by the "past" of the stochastic process, i.e., $\mathcal{F}_{t}^{\vec{x}}=\sigma\{\varphi(s,\omega)\x, s\le t \}$.
	The Markov property implies that for $s \le t$ and every random variable $Y$, measurable with respect to $\mathcal{F}_{t}^{\vec{x}}$, the following relation holds
	\begin{equation}\label{eq:MarkovProperty}
	\mathbb{E}[Y | \mathcal{F}_{s}^{\vec{x}}] = \mathbb{E}[Y | \varphi(s,\omega)\vec{x}].
	\end{equation}
	Moreover, for such $Y$ the following equality, known as the Chapman-Kolmogorov equation, is valid 
	\begin{equation}\label{eq:ChKolmEq}
	\mathbb{E}\left[Y | \varphi(s,\omega)\vec{x} \right] = \mathbb{E}\left[ \mathbb{E}\big[Y | \varphi(t,\omega)\vec{x} \big] \big| \varphi(s,\omega)\vec{x} \right].
	\end{equation}
	The Chapman-Kolmogorov equation implies the semigroup property of the stochastic Koopman operators, i.e.,
	\begin{equation}\label{eq:Semigroup}
	\mathcal{K}^{t+s}=\mathcal{K}^{s}\circ\mathcal{K}^{t}.
	\end{equation}
	For the settings of Markov process, the Koopman operator is known under the name Markov propagator or transition operator and its properties have been studied for a long time \cite{Dynkin,Yosida}.  }

{
	There are two important and well known classes of RDS, which can be identified with the Markov family. These are the discrete-time RDS generated by an independent identically distributed process and the continuous-time RDS generated by the stochastic differential equations where the noise is modeled using Wiener process. In each of these cases the probability space associated with the stochastic process modeling the noise can be identified with the canonical measure-preserving dynamical system (\cite[Appendix A.2]{Arnold}). }

In what follows we describe briefly the canonical dynamical system $(\theta(t))_{t \in \mathbb{T}}$ that is induced by some given stochastic process $\tilde{\xi}$. { We will use this concept in the next subsection.} Let suppose that $\tilde{\xi} = (\tilde{\xi}_t)_{t \in \mathbb{T}}$, $\tilde{\xi}_t : \tilde{\Omega} \rightarrow B$ is a $B$ valued stochastic process on a probability space $\tilde{\Omega}$, where $(B, \mathcal{B})$ is a measurable state space. The given process and the probability measure on $\tilde{\Omega}$ induce a probability measure $P$ on $B^{\mathbb{T}}$, so that one can define a new probability space $(\Omega,\mathcal{F},P)=(B^{\mathbb{T}},\mathcal{B}^{\mathbb{T}},P)$ where $\mathcal{B}^{\mathbb{T}}$ is the $\sigma$-algebra generated by the collection of the cylinder sets. The canonical realization $\xi$ of the stochastic process $\tilde{\xi}$ is defined on $(\Omega,\mathcal{F},P)$ by the coordinate functions $\xi_t({\omega})={\omega}(t)$, $\omega \in \Omega$. 
The shift transformations $\theta(t) : B^{\mathbb{T}} \rightarrow B^{\mathbb{T}}$ given by
\begin{equation}\label{eq:shift}
\theta(t) \omega(\cdot) := \omega(t+\cdot), \quad t \in \mathbb{T},
\end{equation}
constitute the semigroup or group of measurable transformations.
Notice that the canonical realization $\xi_t({\omega})$ can be viewed as the composition of the shift transformation (\ref{eq:shift}) and of the canonical projection  $\pi : B^{\mathbb{T}} \rightarrow B$ defined by 
\begin{equation}\label{eq:CanonProjection}
\pi(\omega(\cdot)) = \omega(0),
\end{equation}
i.e., $\xi_t(\omega)=\omega(t)=\pi(\theta(t)\omega)$.

{If time $\mathbb{T}$ is discrete 
	($\mathbb{T}=\mathbb{Z}$ or $\mathbb{T}=\mathbb{Z}^{+}$) a map $(t,\omega) \rightarrow \theta(t)\omega$ is measurable and $(\theta(t))_{t \in \mathbb{T}}$ is a measurable dynamical system. In the continuous time case ($\mathbb{T}=\mathbb{R}$ or $\mathbb{T}=\mathbb{R}^{+}$), when $B=\mathbb{R}^{m}$ or when $B$ is a Polish space, $(\theta(t))_{t \in \mathbb{T}}$ could become a measurable dynamical system after some redefinition of the probability measure $P$ and of the $\sigma$-algebra set (see \cite{Arnold}, Appendix A.2).
}

\subsection{Discrete-time RDS}
Let now assume that for $\mathbb{T} = \mathbb{Z}^{+} \cup \{0\}$, { $\omega=(\omega_i)_{i\in\mathbb{T}}$ is a canonical realization of the stochastic process with the associated driving system composed by the  shift transformation maps (\ref{eq:shift}) as described in the previous paragraph. }
If we assume that the discrete RDS $\varphi(n,\omega)$ is defined by the one step map $T(\omega,\cdot) : M \rightarrow M$ of the form 
{
	\begin{equation}\label{eq:T_Markov}
	T(\omega,\cdot)=T_0(\pi(\omega),\cdot),
	\end{equation}
	where $\pi$ denotes the canonical projection (\ref{eq:CanonProjection}),
	by taking into account (\ref{eq:T_Markov}) in (\ref{eq:DRDS_Cocycle}), we get
	\begin{equation}\label{eq:DRDS_Cocycle_Markov}
	\varphi(n,\omega) = T_0(\pi(\psi^{n-1}(\omega)),\cdot) \circ \cdots \circ T_0(\pi(\psi(\omega)),\cdot) \circ T_0(\pi(\omega),\cdot), \ \ n \ge 1.
	\end{equation} 
	If $\omega$ is an i.i.d. stochastic process, the sequence  $\{T(\psi^{i}(\omega),\cdot)=T_0(\pi(\psi^{i-1}(\omega)),\cdot),\\ i \ge 1\}$ is an i.i.d. sequence of random maps, so that RDS (\ref{eq:DRDS_Cocycle_Markov}) generates the time-homogeneous Markov family $\{\varphi(n,\omega)\x, \x \in M\}$ \cite[Chapter 2.1]{Arnold}. 
	In this case the stochastic Koopman operator family is a semigroup, so that $\mathcal{K}^{n} = \mathcal{K}^{1} \circ \cdots \circ \mathcal{K}^{1}= (\mathcal{K}^{1})^{n}$, $n > 0$. Therefore one can think about $\mathcal{K}^{1}$ as the generator of the stochastic Koopman semigroup, and we denote it by $\mathcal{K}^{S}$. According to (\ref{eq:KoopmanStochasticGeneral}), $\mathcal{K}^{1}$ is determined by using the one step map $T(\omega,\cdot)$ as 
	\begin{equation}\label{eq:KoopmanStochasticDiscrete}
	\mathcal{K}^{S} f(\x)= \mathcal{K}^{1} f(\x) = \mathbb{E} \left[f (T(\omega,\x))\right].
	\end{equation}
	It follows from the semigroup property that if $\lambda^S$ is the eigenvalue of the stochastic Koopman generator with the associated eigenfunction $\phi(x)$, then $(\lambda^{S})^{n}$ and  $\phi(x)$ are the eigenvalue and the eigenfunction of the operator $\mathcal{K}^{n}$. }
{
	\begin{rem}
		Note that in the more general case presented in subsection \ref{subsec21:DiscreteRDS}, the random maps were identically distributed but not necessarily independent, so that the future state of the process obtained by the action of their composition (\ref{eq:DRDS_Cocycle}) could depend on the past behavior of the system and not only on the present state. 
	\end{rem} 
}

\color{black}
\begin{exa}[Noisy rotation on the circle]\label{nrot}
	We describe here the example considered in \cite{JungeMarsdenMezic}. A deterministic dynamical system representing the rotation on the unit circle $S^{1}$ is defined by
	\begin{equation}\label{eq:RotMap}
	T(x)=x+\vartheta,
	\end{equation}
	where $\vartheta \in S^{1}$ is constant number.
	We consider here its stochastic perturbation, i.e., { a discrete RDS over the dynamical system $\theta=(\theta(t))_{t \in \mathbb{T}}$, where $\theta(t)$ are shift transformations (\ref{eq:shift}),} defined by the one step map $T:\Omega \times S^{1} \rightarrow S^{1}$ of the form (\ref{eq:T_Markov}):
	\begin{equation}\label{eq:RotStochMap}
	T(\omega,x)=x+\vartheta+\pi(\omega).
	\end{equation}
	Here $\omega \in [-\delta/2,\delta/2]^{\mathbb{Z}}$ is a canonical realization of a stochastic process and $\pi(\cdot)$ is the canonical projection defined by (\ref{eq:CanonProjection}). We suppose that the coordinates $\omega_i$ are i.i.d. with uniform distribution on the interval $[-\delta/2,\delta/2]$ for some $\delta > 0$. 
	According to (\ref{eq:KoopmanStochasticDiscrete}), the action of the associated stochastic  Koopman generator on an observable function $f : S^{1} \rightarrow \mathbb{C}$ is given by 
	\begin{equation}\label{eq:Koopman_rotation}
	\mathcal{K}^{S} f(x) = \mathbb{E} [f(T(\omega,x))] 
	= \frac{1}{\delta} \int_{-\delta/2}^{\delta/2}  f(x+\vartheta+\omega_0) d\omega_0.
	\end{equation}
	For the functions
	\begin{equation}\label{eq:SRotEF}
	\phi_j(x) = \exp{(i 2\pi j x)}, \ \ j \in \mathbb{Z},
	\end{equation} 
	the following equality holds
	\begin{align}\label{eq:Koopman_rotation_eig}
	\mathcal{K}^{S} \phi_j(x) &= \frac{1}{\delta} \int_{-\delta/2}^{\delta/2}  \exp(i 2\pi j (x+\vartheta+\omega_0)) d\omega_0 \nonumber \\
	& = \frac{\sin{(j \pi  \delta)}}{j \pi\delta} \exp(i 2\pi j \vartheta) \exp(i 2\pi j x) \nonumber \\
	&= \frac{\sin{(j \pi  \delta)}}{j \pi\delta} \exp(i 2\pi j \vartheta)  \phi_j(x).
	\end{align}
	We easily conclude that (\ref{eq:SRotEF}) are the eigenfunctions of the stochastic Koopman generator with corresponding eigenvalues
	\begin{equation}\label{eq:SRotEV}
	\lambda^{S}_j = \frac{\sin{(j \pi  \delta)}}{j \pi\delta} \exp(i 2\pi j \vartheta), \ \ j\in {\mathbb {Z}}
	\end{equation}
	{For any function  $f:L^2(S^1)\rightarrow \mathbb{C}$ we have the spectral expansion
		\begin{equation}
		{\cal K}^n f(x)=\sum_{j\in \mathbb{Z}} c_j \left(\frac{\sin{(j \pi  \delta)}}{j \pi\delta}\right)^n \exp(i 2\pi j n \vartheta)\exp(i 2\pi j x).
		\end{equation}
		where $c_j$ are the Fourier coefficients of $f$. Clearly, ${\cal K}^n f(x)\rightarrow c_0$ as $t\rightarrow\infty$. }
	It is known that the eigenvalues of the Koopman generator related to the deterministic dynamical system (\ref{eq:RotMap}) lie on the unit circle and are equal to $\lambda_j = \exp(i 2\pi j \vartheta),$ while the eigenfunctions are the same as in the random case. 
	Moreover, it is interesting to observe that for rational $\vartheta$ the eigenspaces in deterministic case, i.e., for $\delta = 0$, are infinite dimensional, while they {are finite dimensional in the stochastic case ($\delta > 0$) due to the compactness of the operator. To be more precise, in this example, the eigenspaces become one dimensional.} 
	
\end{exa}

\begin{rem}
	Consider the case when randomness is additive, i.e., when the one-step map $T : \Omega \times \mathbb{R}^{d} \rightarrow \mathbb{R}^{d}$ is given by
	\[
	T(\omega, \x)=\A \x+\pi(\theta(t)\omega),
	\]
	where $\theta(t)$ is shift transformation and {$\omega : \Omega \rightarrow \mathbb{R}^{d}$ is the canonical realization of a process with  i.i.d. components.
		Suppose that $\mathbb{E} [\omega] = 0$} and that the matrix $\A$ is diagonalizable with simple eigenvalues $\lambda_j, \,  j=1,\ldots,d$.
	Then the eigenfunctions of the stochastic Koopman generator $\mathcal{K}^{S}$ are principal eigenfunctions of the
	form $\phi_j(\x)=\langle \x,  {\bw}_j \rangle,  \ \ j=1,\ldots,d$, where $\bw_j, \,  j=1,\ldots,d$ are left eigenvectors of $\A$, while its eigenvalues coincide with the eigenvalues of the matrix $\A$.
\end{rem}

\subsubsection{Continuous-time RDS}

Let suppose that the stochastic Koopman operators family satisfies semigroup property.
Define the generator of the stochastic Koopman family $(\mathcal{K}^{t})_{t \in \mathbb{T}}$ acting on the observable functions {$f \in {C}_b^{1}(\mathbb{R}^{d})$  (${C}_b^{1}(\mathbb{R}^{d})$ is the space of continuously differentiable functions on $\mathbb{R}^{d}$ with bounded and continuous first derivative)} by the limit 
\begin{equation}\label{eq:KoopmanGeneratorLimit}
\mathcal{K}^{S} f(\x) = \lim_{t \rightarrow 0 +} \frac{\mathcal{K}^{t} f(\x) - f(\x)}{t},
\end{equation}
if it exists.
For the Koopman operators associated to the RDS generated by RDE we have the following proposition.

\begin{propos}\label{prop:KoopmanGenRDE}
	If the solution of RDE (\ref{eq:RDE}) is differentiable with respect to $t$ and  the stochastic Koopman family $(\mathcal{K}^{t})_{t \in \mathbb{T}}$ is a semigroup, then the action of the generator $\mathcal{K}^{S}$ 
	{ on $f \in {C}_b^{1}(\mathbb{R}^{d})$} is equal to
	\begin{equation}\label{eq:KoopmGenRDEequation}
	\mathcal{K}^{S} f(\x)= \mathbb{E} \left[ F(\omega,\x) \right] \cdot \nabla f(\x).
	\end{equation}
\end{propos}
\begin{proof}
	\begin{align}\label{eq:RDES_KoopmanGen}
	& \mathcal{K}^{S} f(\x) = \lim_{t \rightarrow 0 +} \frac{\mathcal{K}^{t} f(\x) - f(\x)}{t}  = \lim_{t \rightarrow 0 +} \frac{\mathbb{E} \left[ f (X(t,\omega,\x)) \right] - f(\x)}{t}  \nonumber \\
	& = \lim_{t \rightarrow 0 +} \frac{ \mathbb{E}\left[ f(X(t,\omega,\x)) -f(\x)\right ]}{t} = \mathbb{E} \left[ \lim_{t \rightarrow 0 +}  \frac{f(X(t,\omega,\x)) -f(X(0,\omega,\x))}{t} \right]   \nonumber \\
	& =  
	\mathbb{E} \left[ \frac{d}{dt}  f(X(t,\omega,\x))  \bigg|_{t=0} \right] 
	= \mathbb{E} \left[ \nabla f(\x) \cdot \frac{d}{dt} X(t,\omega,\x) \bigg|_{t=0} \right] \nonumber \\
	&  =  \mathbb{E} \left[ F(\omega,\x) \right] \cdot \nabla f(\x).
	\end{align}
	{ The swapping  of the order of limit and expectation in the second line is justified by the dominated convergence theorem and the fact that the convergence $\displaystyle \frac{f(X(t,\omega,\x)) -f(\x)}{t}$ is uniform for all $\omega$ and $x$, since the derivative of $f$ is bounded and the solution is differentiable (see \cite[Section 7.6]{LasotaMackey} for the proof in the deterministic case).  }
\end{proof}
\begin{corol}\label{RDE_ef_product}
	Suppose that a stochastic Koopman generator $\mathcal{K}^{S}$ associated to RDE (\ref{eq:RDE}) exists.
	If $\phi_1$ and $\phi_2$ are the eigenfunctions of $\mathcal{K}^{S}$ with the associated eigenvalues $\lambda_1$ and $\lambda_2$, then $\phi_1 \phi_2$ is also an eigenfunction with the associated eigenvalue $\lambda_1+\lambda_2$.
\end{corol}
\begin{proof} 
	Since $\mathcal{K}^{S} \phi_i(\x) = \lambda_i \phi_i(\x) = \mathbb{E} \left[ F(\omega,\x) \right] \cdot \nabla \phi_i(\x) $, for $i=1,2$, we have
	\begin{align}
	\mathcal{K}^{S} (\phi_1   \phi_2)(\x)& =  \mathbb{E} \left[ F(\omega,\x) \right] \cdot \nabla (\phi_1   \phi_2)(\x) \nonumber \\
	&= 
	\mathbb{E} \left[ F(\omega,\x) \right]  \cdot \nabla \phi_1(\x) \cdot \phi_2(\x) +
	\mathbb{E} \left[ F(\omega,\x) \right] \cdot \nabla \phi_2(\x) \cdot \phi_1(\x)\nonumber \\
	&=
	\lambda_1 \phi_1(\x) \phi_2(\x) + \lambda_2 \phi_2(\x) \phi_1(\x) = (\lambda_1+\lambda_2) (\phi_1 \phi_2) (\x) \nonumber
	\end{align}
\end{proof}

\begin{corol} Let $\phi \in {C}_b^{1}(\mathbb{R}^{d})$ be an eigenfunction associated with eigenvalue $\lambda$ of the stochastic Koopman generator $\mathcal{K}^{S}$ associated with an RDE (\ref{eq:RDE}). Then
	\begin{equation}
	\frac{d}{dt} \phi(\varphi(t,\omega)\x) =\lambda \phi(\varphi(t,\omega)\x)+\tilde{F}(\omega,\x)\cdot\nabla\phi(\varphi(t,\omega)\x),
	\end{equation}
	where 
	\[
	\tilde{F}(\omega,\x)=F(\omega,\x)-\mathbb{E} \left[ F(\omega,\x) \right].
	\]
\end{corol}

\begin{corol}\label{LRDE_gen}
	Suppose that a stochastic Koopman generator $\mathcal{K}^{S}$ associated with linear RDE (\ref{eq:LRDE}) exists. Also, assume that $\hat{\A} = \mathbb{E} [\A(\omega)]$ is diagonalizable, with simple eigenvalues $\hat{\lambda}_j$ and left eigenvectors $\hat{\bw}_j$, $j=1,\ldots,d.$ Then 
	\begin{equation}\label{eq:LRDE_gen_ef}
	\phi_j(\x)=\langle \x,  \hat{\bw}_j \rangle,  \ \ j=1,\ldots,d,
	\end{equation}
	are the principal eigenfunctions of the generator $\mathcal{K}^{S}$, while $\lambda_j^{S}=\hat{\lambda}_j$ are the associated principal eigenvalues.
\end{corol} 

\begin{proof}
	According to (\ref{eq:RDES_KoopmanGen}), the action of the generator $\mathcal{K}^{S}$ of the Koopman operator family associated with the linear RDS generated by (\ref{eq:LRDE}) { on $f \in C_{b}^{1}(\mathbb{R}^{d})$} is equal to
	\begin{align}\label{eq:linRDES_KoopmanGen}
	\mathcal{K}^{S} f(\x) = \mathbb{E}[\A(\omega) \x] \cdot \nabla f(\x) = \hat{\A}  \x \cdot \nabla f(\x). 
	\end{align}
	Thus
	$$\mathcal{K}^{S} \phi_j(\x) = \langle \hat{\A} \x, \bw_j \rangle = \langle \x, \hat{\A}^{*} \bw_j \rangle = \hat{\lambda}_j\, \phi_j(\x),$$
	which proves the statement.
\end{proof}
\begin{rem}
	Provided that the assumptions of Corollary \ref{LRDE_gen} are valid, the principal eigenfunctions  $\phi_j(\x)$ given by (\ref{eq:LRDE_gen_ef}) are the eigenfunctions of each Koopman operator $\mathcal{K}^{t}$ also, with the corresponding principal eigenvalues $\lambda_j^{S}(t)=\mbox{e}^{\hat{\lambda}_j t}$. { The set of principal eigenfunctions does not cover all the eigenfunctions of the Koopman operator as we discuss next.}
	According to Corollary \ref{RDE_ef_product}, { over the space of real analytic functions,}
	$$\phi(\x) = \phi_1^{n_1}(\x) \cdots \phi_d^{n_d}(\x), \quad \lambda = \sum_{j=1}^{d} n_j \hat{\lambda}_j,$$
	with $n_j \in \mathbb{N}^{+} \cup \{0\}, j=1,\ldots,d$, are the eigenvalues and eigenfunctions of the Koopman generator. Thus, like in deterministic case, any analytic observable function $f$ can be represented as linear combination of powers of the principal eigenfunctions \cite{mezic:2016} and its evolution under the RDS can be obtained using spectral expansion formula.  
\end{rem}

Another type of RDS which could be identified with the time-homogeneous Markov family are the RDS generated by the SDE of the form (\ref{eq:SDE_Nonaut}) with autonomous functions $G$ and $\sigma$, i.e.,
\begin{equation}\label{eq:SDE}
dX_t = G(X_t) dt + {\sigma}(X_t) dW_t.
\end{equation} 
In this case the stochastic differential equation generates the one-parameter family of RDS $\varphi(t, \omega) : = \varphi(t,0,\omega) = \varphi(t+t_0,t_0,\omega)$, so that the corresponding stochastic Koopman operator family and the associated stochastic eigenvalues and eigenfunctions depend on parameter $t$, only.  In this autonomous setting, we denote by $X_{t}(\x)$ the solution of equation (\ref{eq:SDE}) for the initial condition $X_0(\omega)=\x$. In accordance with (\ref{eq:SDE_xw}) it is equal to
\begin{equation}\label{eq:SDE_x}
X_{t}(\x) = \x + \int_{0}^{t} G(X_{s}) ds + \int_{0}^{t} {\sigma}(X_{s}) dW_s.
\end{equation}
There are numerous books dealing with the properties of the solutions of such SDE, {for example \cite[Chapter 16]{CohenElliott}, \cite{DaPratoZabczyk},\cite{Pavliotis}. Here we reference some of the books containing proofs that the generated solutions form Markov family and moreover, that the associated transfer operators form semigroup: \cite{ArnoldSDE}, \cite[Section 2.3]{Arnold}, \cite[Chapter 11]{LasotaMackey}, \cite[Chapter 17]{CohenElliott}, \cite[Section 9]{DaPratoZabczyk}. }

{
	In most of the stated references the generator of the stochastic solutions $X_t$ is studied. 
	This generator is in the stochastic literature known under the name backward Kolmogorov operator and is a second-order elliptic type differential operator. It appears as the right-hand side of the backward Kolmogorov equation, whose solution, with the initial condition prescribed as the value of the observable function $f(\x)$, gives the time evolution of the average value of the observable function, i.e., the value $\mathcal{K}^{t}f(\x)$ (\cite[Sections 2.5, 3]{Pavliotis} \cite[Section 17.4]{CohenElliott}, \cite[Section 9.3]{DaPratoZabczyk}). Thus, the Kolmogorov operator is also the generator of the stochastic Koopman semigroup, and its action over the space of real bounded, continuous, and twice differentiable functions on $\mathbb{R}^{d}$ with bounded and continuous first and second derivatives, which we denote by ${C}_b^{2}(\mathbb{R}^{d})$, is given in the following proposition. 
}

\begin{propos}\label{prop:SKEGenerator}
	The action of the generator of the stochastic Koopman family $\mathcal{K}^{S}$ on  { $f \in {C}_b^{2}(\mathbb{R}^{d})$} is given by
	\begin{equation} \label{SKE_Generator-md}
	\mathcal{K}^{S} f(\x) = G(\x) \nabla f(\x) + \frac{1}{2} \mbox{Tr}\left(\sigma(\x) (\nabla^{2} f(\x)) \sigma(\x)^{T}\right), 
	\end{equation}
	where Tr denotes the trace of the matrix. 
\end{propos}
{The following proposition follows from the proof of the Feynman-Kac formula \cite[Section 3.4]{Pavliotis}, \cite[Section 17.4]{CohenElliott}.}
\begin{propos}\label{prop:SDE_ef_evol}
	Let  {$\phi \in {C}_b^{2}(\mathbb{R}^{d})$} be an eigenfunction of the stochastic Koopman generator $\mathcal{K}^{S}$ associated with RDS generated by SDE (\ref{eq:SDE}) with the corresponding eigenvalue $\lambda$. Then
	\begin{equation}
	d\phi (X_t) = 
	\lambda \phi(X_t) dt + \nabla \phi(X_t) \sigma(X_t) dW_t \label{eq:KSeigenDE-md1}.
	\end{equation}
	$\phi$ is the eigenfunction of the stochastic Koopman operator $\mathcal{K}^{t}$ also, i.e.~ 
	\begin{equation}\label{eq:KSeigen3}
	\mathcal{K}^{t} \phi(\x) = \mbox{e}^{\lambda t} \phi (\x).
	\end{equation}
\end{propos}
\begin{proof}
	Suppose $d=1$. It follows from { It\^{o}'s formula 
		\begin{equation}\label{eq:Itoformula}
		dY_t = f'(X_t) dX_t + \frac{1}{2} f''(X_t) \sigma(X_t)^{2} dt
		\end{equation} 
	} that the eigenfunction $\phi(X_t)$ evolves according to 
	\begin{align}
	d\phi (X_t) & = \phi'(X_t) G(X_t) dt + \frac{1}{2} \phi''(X_t) \sigma(X_t)^{2} dt  + \phi'(X_t) \sigma(X_t) dW_t \nonumber \\
	& =  \mathcal{K}^{S} \phi(X_t) dt + \phi'(X_t) \sigma(X_t) dW_t \nonumber \\
	& = 
	\lambda \phi(X_t) dt + \phi'(X_t) \sigma(X_t) dW_t, \label{eq:KSeigenDE}
	\end{align}
	where in the last equality we used that $\mathcal{K}^{S} \phi(x) = \lambda \phi(x)$. By using the similar procedure, the equations (\ref{eq:KSeigenDE-md1}) valid in  multidimensional case can be easily derived.
	
	{ The fact that $\phi$ is an eigenfunction of each Koopman semigroup member $\mathcal{K}^{t}$
		follows from the spectral mapping theorem \cite[Chapter IV.3]{EngelNagel}.}
\end{proof}


The fact that the eigenfunctions related to the generator of the dynamical system evolve according to (\ref{eq:KSeigenDE}), is used in \cite{ShnitzerTalmonSlotine}. In that paper the approximations of eigenfunctions computed via diffusion maps are used to parametrize the state of the RDS using the linear operator similarly as in Koopman framework.

\begin{rem}
	Unlike in the case of RDS generated by the RDE, the product of eigenfunctions of the stochastic Koopman generator associated to RDS generated by SDE is not necessarily an eigenfunction. This easily follows from Proposition \ref{prop:SKEGenerator}. However, the eigenfunctions in many cases satisfy a recurrence relationship (e.g. Hermite polynomials, which are Koopman eigenfunctions of { Ornstein-Uhlenbeck processes obtained as a solution of the Ornstein-Uhlenbeck SDE \cite[Section 4.4]{Pavliotis}}) and thus can be deduced from the principal eigenfunctions. This reduces the problem to analysis of principal eigenfunctions, and thus $d$ objects in an $d$ dimensional space, remarkable for a nominally infinite-dimensional representation.
\end{rem}

\section{Numerical approximations of the stochastic Koopman operator} \label{sec:NumSKO}
One of the goals of this work is to compute the numerical approximations of the spectral objects of the stochastic Koopman operators, i.e., their eigenvalues and eigenfunctions. This is performed by extending the DMD algorithms, that are originally developed for approximating the spectral object of the Koopman operators in deterministic settings, to the stochastic framework. { DMD algorithms and the spectral analysis of the Koopman operator are connected in the following way. Suppose that the restriction of the infinite dimensional stochastic Koopman operator $\mathcal{K}^{t}$ to an appropriate $n$ dimensional subspace of functions is closed under the action of the Koopman operator (see \cite{Drmac2017}). Let $\mathbb{K} \in \mathbb{C}^{n \times n}$ denote its representation by an $n \times n$ dimensional matrix. The goal of the DMD numerical algorithms is to determine the spectrum and associated eigenvectors of the finite dimensional linear operator $\mathbb{K}$, and those are in turn the eigenvalues and coordinates of eigenfunctions of the associated Koopman operator in the chosen finite-dimensional basis. If the considered subspace is not closed under the action of the Koopman operator, one can expect that under certain assumptions, the operator $\mathbb{K}$ at least approximates the underlying dynamics \cite{ArbabiMezic2016,KordaandMezic:2017}. }

{ The DMD algorithm provide us with  a decomposition of the pair $(\X_m,\Y_m)$ given by the eigenvalues and the eigenvectors of the operator $\mathbb{K}$, which is in the data-driven settings not known explicitly, while it is known that its action on the range $\X_m$ should be equal to  $\Y_m$ \cite{Tu2014jcd}. For $n \gg m$ and full column range of $\X_m$ there exists an exact solution of the equation $\Y_m = \mathbb{K} \X_m$, while for $m \gg n$ the equation is satisfied in a least squares sense \cite{Tu2014jcd}.}

Here we use the enhanced DMD algorithm, proposed in \cite{Drmac2017}, which we denote hereafter by DMD RRR. We briefly describe the algorithm and propose two approaches for applying it to RDS: the standard DMD approach using snapshot pairs, and the DMD applied to the stochastic Hankel matrix (sHankel-DMD), where time-delayed snapshots are used.

\subsection{The DMD algorithms for RDS}\label{subsec:DMDRRR}
Let $\f = (f_1,\ldots,f_n)^{T} : M \rightarrow \mathbb{C}^{n}$ be a vector valued observable on the state space. For an $\x \in M$, let $\f^{k}(\omega,\x) = \f (T^{k}(\omega, \x)), k=0,1,\ldots$  be a vector-valued observable series on the trajectory of the considered discrete-time RDS. Denote its expectation by $\f^{k}(\x) = \mathbb{E}[\f^{k}(\omega,\x)]=\mathcal{K}^{k} \f(\x)$.
For the continuous-time RDS $\varphi$, we choose the time step $\Delta t$ and define the values of the observables in the same way as in the discrete-time case by taking  $T^{k}(\omega,\x)=\varphi(k \Delta t,\omega) \x$. 
The value of the observable at the time moment $t_k= k \Delta t$ is evaluated as $\f^{k}(\x) = \mathcal{K}^{k}_{\Delta t} \f(\x) = \mathbb{E}[\f(\varphi(k \Delta t,\omega) \x)$. 

{
	In this stochastic framework, the DMD algorithm is applied to the matrices { $\X_m, \Y_m \in \mathbb{C}^{n \times m}$, which are for the chosen initial states $\x_1,\ldots,\x_m \in M$} defined by 
	\begin{equation}\label{eq:DMDRRR_XY_1}
	\X_m = \begin{pmatrix} \f^{0}(\x_1) & \f^{0}(\x_2) & \ldots &\f^{0}(\x_m) \end{pmatrix} \ \text{and} \ 
	\Y_m = \begin{pmatrix} \f^{k}(\x_1)  & \f^{k}(\x_2)& \ldots & \f^{k}(\x_m) \end{pmatrix}.
	\end{equation} 
	In this case we expect that the DMD algorithm provides us with the eigenvalues and eigenfunctions of the finite dimensional approximation operator associated with the Koopman operator $\mathcal{K}^{t_k}$. In the case when the Koopman operator family is a semigroup, we have $\mathcal{K}^{k}_{\Delta t}=(\mathcal{K}^{1}_{\Delta t})^{k}$, thus it makes sense to apply the DMD algorithm to the time-delayed snapshots, which means that for the chosen initial condition $\x_0$, we define the matrices $\X_m$ and $\Y_m$ as 
	\begin{equation}\label{eq:DMDRRR_XY_2}
	\X_m = \begin{pmatrix} \f^{0}(\x_0) & \f^{1}(\x_0) & \ldots &\f^{m-1}(\x_0) \end{pmatrix} \quad \text{and} \quad \Y_m = \begin{pmatrix} \f^{1}(\x_0) & \f^{2}(\x_0) & \ldots &\f^{m}(\x_0) \end{pmatrix}.
	\end{equation}  
	In this case, we expect that the algorithm provide us with the eigenvalues and eigenfunctions of the approximation operator associated with $\mathcal{K}^{1}_{\Delta t}$.
	The proof of the convergence of the DMD type algorithm with input matrices (\ref{eq:DMDRRR_XY_2}) where $f_1,\ldots,f_n$ span the finite dimensional invariant subspace of the stochastic Koopman operator, to its eigenvalues and eigenfunctions is given in \cite{TakeishiKawaharaYairi} for the RDS in which the noise is modeled by i.i.d. random variables and under the assumption of ergodicity.} 

We describe now the crucial steps of the DMD RRR algorithm. The DMD RRR algorithm starts in the same way as other SVD-based DMD algorithms, i.e.,~with the SVD decomposition for low dimensional approximation of data: $\X_m = \U \boldsymbol{\Sigma} \V^{*} \approx \U_r \boldsymbol{\Sigma}_r \V_r^{*}$, where $\boldsymbol{\Sigma} = \mbox{diag} (\left(\sigma_i \right)_{i=1}^{\min(m,n)}$,   $\sigma_i$ are singular values  arranged in the descending order, i.e. $\sigma_1 \ge \cdots \ge \sigma_{\min(m,n)} \ge 0$, and $r$ is the dimension of the approximation space. Then { $\Y_m = \mathbb{K} \X_m $} can be approximated as
\begin{equation}\label{eq:SVD}
\Y_m \approx \mathbb{K} \U_r \boldsymbol{\Sigma}_r \V_r^{*}.
\end{equation}
Since $\mathbb{K} \U_r = \Y_m \V_r \boldsymbol{\Sigma}_r^{-1}$, the { Rayleigh quotient matrix} $\mathbf{S}_r = \U_r^{*} \mathbb{K} \U_r$ with respect to the range $\U_r$ can be computed in this data-driven setting as  
\begin{equation}\label{eq:RayleighQ_S}
{ \mathbf{S}_r =  \U_r^{*} \Y_m \V_r \boldsymbol{\Sigma}_r^{-1}.}
\end{equation}
Each eigenpair $(\lambda, \bw)$ of $\mathbf{S}_r$ generates the corresponding Ritz pair $(\lambda, \U_r \bw)$, that is a candidate for the  approximation of the eigenvalue and eigenvector of the Koopman operator. 

Here we emphasize few crucial points at which the DMD RRR algorithm is improved in comparison with the standard DMD algorithms \cite{schmid:2010,Tu2014jcd}. The first point of difference refers to the dimension $r$ of the reduction space. Instead of defining the dimension of the space a priori or to take into account the spectral gap  in singular values, it is proposed in \cite{Drmac2017} to take into account user supplied tolerance $\varepsilon$, which is then used for defining $r$ as the largest index satisfying $\sigma_r \ge \sigma_1 \varepsilon.$ 
The algorithm is, according to \cite{Drmac2017}, further enhanced with the residual computation for each Ritz vector pair $(\lambda, \U_r \bw)$ 
\begin{equation}\label{eq:DMDresidual}
{	\eta = \|\mathbb{K}(\U_r \bw) - \lambda (U_r \bw)\|_2 = \| (\Y_m \V_r \boldsymbol{\Sigma}_r^{-1})\bw - \lambda (U_r \bw)\|_2,}
\end{equation}
where $\| \cdot \|_2$ stands for $L_2$ norm. Then the vectors at which the required accuracy is not attained are not taken into account. The final improvement compared to the standard algorithms refers to scaling of the initial data, i.e., if the matrix $D_x$ is defined with $D_x=\mbox{diag}(\|\X_m(:,i)\|_2)_{i=0}^{m-1}$, we set $\X_m = \X_m^{(1)} D_x$  and $\Y_m = \Y_m^{(1)} D_x$ and proceed with $\X_m^{(1)}$ and $\Y_m^{(1)}$ as data matrices.

\subsection{Stochastic Hankel DMD (sHankel-DMD) algorithm}

Here we define the stochastic Hankel matrix and describe the application of the DMD type algorithm on it.  
We use the Hankel DMD method described in \cite{ArbabiMezic2016} for deterministic dynamical system and expand the ideas to the stochastic settings. { We first limit to the discrete time RDS and assume that the associated stochastic Koopman operator family satisfies the semigroup property, since only in such case the results of the DMD algorithm applied to the Hankel matrix are meaningful. Otherwise, the situation similar to the nonautonomous deterministic case arise and the application of the DMD algorithm to the time-delayed snapshots becomes an attempt to approximate different operators, which results with significant errors \cite{MacesicZicMezic}.} 

The Hankel matrix in the stochastic framework is defined as follows. 
For a scalar observable $f : M \rightarrow \mathbb{C}$, we define the vector of $n$ observations along the trajectory that starts at $\x \in M$ and is generated by the one step discrete random map $T$:
\begin{equation}\label{eq:fn_xw}
{\f_n (\omega, \x) = \left( f(\x), f(T(\omega,\x)), \ldots,  f (T^{n-1}(\omega,\x)) \right)^{T}.}
\end{equation}
{
	Let $\f_n^{k}(\x)$, $k= 0,1,\ldots$ denote the expectation of $\f_n (\theta(k)\omega, T^{k}(\omega, \x))$ over the trajectories of length $k$, i.e., 
	\begin{align}\label{eq:Efn_x}
	\f_n^{k}(\x) &= \mathbb{E} \left[ \f_n (\theta(k)\omega, T^{k}(\omega, \x)) \right] \nonumber \\
	&=\left( \mathcal{K}^{k} f(\x), \mathcal{K}^{k} f (T(\omega,\x)), \ldots, \mathcal{K}^{k}  f(T^{n-1}(\omega,\x)) \right)^{T}. 
	\end{align}
	Observe that the components of $\f_n^{k}(\x)$ are the values of the function $\mathcal{K}^{k} f$ along the trajectory of length $n$ starting at $\x \in M$.
}

{ The stochastic Hankel matrix of dimension $n \times m$, associated with the trajectories starting at $\x \in M$ and generated by the map $T$ is defined by
	\begin{align}\label{eq:Hankel_sto1}
	& \mathbf{H}_{n \times m}(\omega,\x)  = \left(\f_n^{0}(\x) \ \ \f_n^{1}(\x) \ \ldots \ \f_{n}^{m-1}(\x)\right) 
	\nonumber \\
	& = 
	\left( \begin{array}{cccc}
	f(\x)     &  \mathcal{K}^{1} f(\x) & \ldots  & \mathcal{K}^{m-1} {f(\x)} \\ 
	f(T(\omega,\x))& \mathcal{K}^{1} f(T(\omega,\x)) &  \ldots  & \mathcal{K}^{m-1}f(T(\omega,\x)) \\ 
	\vdots & \vdots & \ddots &  \vdots \\ 
	f(T^{n-1}(\omega,\x))& \mathcal{K}^{1} f(T^{n-1}(\omega,\x)) &\ldots & \mathcal{K}^{m-1} {f(T^{n-1}(\omega,\x)) }
	\end{array} \right). 
	\end{align}}
The {columns} of $\mathbf{H}_{n \times m}$ are approximations of functions in the  Krylov subspace 
\begin{equation}\label{eq:KrylovSubspace}
\mathbb{K}_m(\mathcal{K},f)={\mbox{span}}\begin{pmatrix}
f, & \mathcal{K}^{1} f, & \ldots, & \mathcal{K}^{m-1}  f
\end{pmatrix},
\end{equation} obtained by sampling the values of functions $\mathcal{K}^{k}  f$, $k=0,\ldots,m-1$ along the trajectory of length $n$ starting at $\x \in M$. 

{
	When the DMD algorithm is applied to the stochastic Hankel matrix, the input data matrix $\X_m$ is defined by taking the first $m$ columns of the Hankel matrix $\mathbf{H}_{n \times (m+1)}(\omega, \x)$, while the data matrix $\Y_m$ is formed from the last $m$ columns of the same matrix. We refer to this methodology as to the stochastic Hankel DMD (sHankel-DMD) algorithm. }

As already mentioned, when we consider the continous-time RDS, we could for the chosen $\Delta t$ associate to it a discrete RDS by defining the one-step map as $T(\omega,\x)=\varphi(\Delta t, \omega) \x.$ Then the Koopman operator $\mathcal{K}^{k}$ should be replaced with the operator $\mathcal{K}_{\Delta t}^{k}$.

\subsection{Convergence of the sHankel-DMD  algorithm}
It was proved in  \cite{ArbabiMezic2016} that for ergodic systems,  under the assumption that observables are in an invariant subspace of the Koopman operator, the eigenvalues and eigenfunctions obtained by the extended DMD algorithm applied to Hankel matrix, converge to the true Koopman eigenvalues and eigenfunctions of the considered system. The convergence of the DMD algorithm with input matrices $\X_m$ and $\Y_m$ defined by (\ref{eq:DMDRRR_XY_2}) to the eigenvalues and the eigenfunctions of the Koopman operator is proved in \cite{TakeishiKawaharaYairi} for the class of RDS in which the noise is modeled by i.i.d. random variables, under the assumption of ergodicity, and of the existence of the finite dimensional invariant subspace. Here we prove that under the same assumptions, {the convergence is accomplished for the sHankel-DMD algorithm}. Our proof is based on the fact that the eigenvalues and the eigenfunctions obtained by DMD algorithm correspond to the matrix that is similar to the companion matrix, which represents the finite dimensional approximation of the Koopman operator in the Krylov basis. We limit the considerations to the discrete-time RDS. 
{
	Suppose that the dynamics on the compact invariant set $A \subseteq M$ is generated by the measure preserving  map $T(\omega, \cdot) : A \rightarrow A$ for each $\omega \in \Omega$. We recall from \cite[Section 1.4]{Arnold} that a probability measure $\nu$ is invariant for RDS $\varphi(n,\omega) = T^{n}(\omega,\cdot)$ if it is invariant for the skew product flow $\Theta(n)(\omega,x) = (\theta(n) \omega, T^{n}(\omega,x))$ generated by $T$ and $\theta(t)$, i.e., if $\Theta(n) \nu = \nu$  and if $\pi_\Omega \nu = \mathbb{P}$ where $\pi_\Omega$ denotes the canonical projection $\Omega \times A \rightarrow \Omega$. Invariant measures always exist for a continuous RDS on compact space $A$ (see \cite[Theorem 1.5.10]{Arnold}. If $A$ is a Polish space, the measure $\nu$ on $\Omega \times A$ could be written as product of measures, i.e., $d \nu (\omega,x)= d \mu_\omega(x) d\mathbb{\mathbb{P}}(\omega)$, i.e., for $f \in L^{1}(\nu)$
	$$\int_{\Omega \times A} f d\nu = \int_\Omega \int_A f(\omega,x) d\mu_{\omega}(x) d\mathbb{P}(\omega).$$
	If the skew product dynamical system $\Theta$ is ergodic on $\Omega \times A$ with respect to some invariant measure $\nu$ and if $\theta(n)$ is ergodic with respect to $\mathbb{P}$, we say that $\varphi$ is ergodic with respect to the invariant measure $\nu$. 	
	Under the assumption of ergodicity of $\Theta$ with respect to the measure $\nu$, the Birkhoff's ergodic theorem states that the time average of observable $f \in L^{2}(\Omega \times A; \nu)$ under $\Theta$ is given by
	\begin{equation}\label{eq:BirkErgTh0}
	\lim_{n \rightarrow \infty} \frac{1}{n} \sum_{k=0}^{n-1}  f (\theta(k) \omega, T^{k} (\omega,\x))  = \int_{\Omega \times A} f(\omega,x) d\nu , \quad \textrm{a.~e. on} \ \ \Omega \times A. 
	\end{equation}   
}

The following proposition shows that under the assumption of ergodicity and Markovian property of RDS, the sHankel-DMD algorithm provide us with approximations of eigenvalues and eigenfunctions that converge to the true Koopman eigenvalues and eigenfunctions.
\begin{propos}
	Suppose that the dynamics on the compact invariant set $A \subseteq M$ is given by the one step map $T(\omega, \cdot) : A \rightarrow A$ for each $\omega \in \Omega$ { and that the associated discrete time RDS $\varphi$ is ergodic with respect to some invariant  measure $\nu$. Assume additionally that the processes $\{\varphi(n,\omega)\x, \x \in A \}$ form a Markov family. Denote by $\mu$ the marginal measure $\mu = \mathbb{E}_\mathbb{P}(\nu)$ on $A$. }
	
	Let the Krylov subspace $\mathbb{K}_m(\mathcal{K}, f)$ span an $r$-dimensional subspace of the Hilbert space $\mathcal{H} = L^{2}(A, \mu)$, with $r < m$, invariant under the action of the stochastic Koopman operator. Then for almost every $\x \in A$, as $n \rightarrow \infty$, the eigenvalues and eigenfunctions obtained by applying DMD algorithm to the first $r+1$ columns of the $n \times (m+1)$ dimensional stochastic Hankel matrix, converge to the true eigenvalues and eigenfunctions of the stochastic Koopman operator.
\end{propos} 
\begin{proof}
	Consider the observables $f : A \rightarrow \mathbb{R}$ belonging to the Hilbert space $\mathcal{H} = L^{2}(A,\mu)$. Due to the ergodicity of $\Theta$, in accordance to the Birkhoff's ergodic theorem, (\ref{eq:BirkErgTh0}) is valid, thus we get
	{
		\begin{equation}\label{eq:BirkErgTh3}
		\lim_{n \rightarrow \infty} \frac{1}{n} \sum_{k=0}^{n-1}  f(T^{k} (\omega,\x))  = \!\! \int_{\Omega \times A} \!\!\!\!\!\!  f(x) d\nu = \int_\Omega \int_{A} \!\! f(x) d\mu_\omega(x) d\mathbb{P}(\omega) = \int_A \!\! f d\mu,
		\end{equation}   
		where the last equality follows from the fact that $\mu = \mathbb{E}_\mathbb{P}(\nu)=\mathbb{E}_\mathbb{P}(\mu_\omega)$.}
	For observables $f, g \in \mathcal{H}$ let the vectors of $n$ observations along the trajectory starting at $\x \in A$ of the RDS generated by the map $T$ be denoted by $\f_{n} (\omega, \x)$ and $\g_{n} (\omega, \x)$ and defined by (\ref{eq:fn_xw}).  If we denote the data-driven inner product by $ < \f_{n}(\omega, \x), \g_{n}(\omega, \x) >$, we have
	\begin{align} 
	& \lim_{n \rightarrow \infty}  \frac{1}{n} \, \left[< \f_{n}(\omega, \x), \g_{n}(\omega, \x) >\right] = \lim_{n \rightarrow \infty} \frac{1}{n} \, \sum_{k=0}^{n-1} { f ( T^{k} (\omega,\x) ) g^{*} ( T^{k} (\omega,\x))  } \nonumber \\
	& = \lim_{n \rightarrow \infty} \frac{1}{n} \sum_{k=0}^{n-1} 
	f g^{*} \left(T^{k} (\omega,\x)\right)  = \int_{A} f g^{*} d\mu = <f , g >_\mathcal{H} \text{  for a.e. \x} \label{eq:Hankel_limit}
	\end{align}   
	with respect to the measure $\mu$. Using the assumption that $\mathbb{K}_m(\mathcal{K},f)$ spans $r$-dimensional subspace of $\mathcal{H}$, which is invariant for the stochastic Koopman operator, the restriction of the Koopman operator to this subspace is finite dimensional and can be realized by an $r \times r$ matrix. {The representation of this matrix in the basis formed by the functions $\begin{pmatrix} f, & \mathcal{K} f, & \ldots, & \mathcal{K}^{m-1}f \end{pmatrix}$ is given with the companion matrix
		\begin{equation}\label{eq:CompanionMatrix}
		{\bf C}=\begin{pmatrix}
		0 & 0 & \ldots & 0 & {c}_0 \\
		1 & 0 & \ldots & 0 & {c}_1 \\
		0 & 1 & \ldots & 0 & {c}_2 \\
		\vdots & \vdots & \ddots & \vdots & \vdots \\
		0 & 0 & \ldots & 1 & {c}_{r-1} \\
		\end{pmatrix},
		\end{equation}
		where the vector ${\bf c} = \begin{pmatrix} c_0, & c_1, & \ldots, & c_{r-1} \end{pmatrix}^{T}$, obtained by using least squares approximation is equal to
		\begin{equation}\label{eq:c}
		{\bf c} = {\bf G}^{-1} \begin{pmatrix} <f, \mathcal{K}^{r} f>_\mathcal{H}, & <\mathcal{K}^{1}, \mathcal{K}^{r} f>_\mathcal{H}, & \ldots, & <\mathcal{K}^{r-1}, \mathcal{K}^{r} f>_\mathcal{H} \end{pmatrix}^{T}.
		\end{equation}
		Here ${\bf G}$ denotes the Gramian matrix with elements $G_{ij} = <\mathcal{K}^{i-1}, \mathcal{K}^{j-1} f>_\mathcal{H}\nolinebreak,$ $i,j = 1,\ldots,r$ (see \cite{ArbabiMezic2016,Drmac2017}).
		
		Consider now the stochastic Hankel matrix $\mathbf{H}_{n \times (r+1)}(\omega,\x)$ of dimension $n \times (r+1)$ along a trajectory starting at $\x$ and the companion matrix algorithm \cite{ArbabiMezic2016,Drmac2017} applied to the matrices  
		$\X_r = \left( \f_{n}^{0}(\x) \ \f_{n}^{1}(\x)  \ \ldots  \ \f_{n}^{r-1}(\x) \right)$ 
		and   $\Y_r = \left(\f_{n}^{1}(\x) \ \f_{n}^{2}(\x)  \ \ldots  \ \f_{n}^{r}(\x) \right)$.  We denote by $\tilde{\bf C}$ the numerical companion matrix 
		computed as a best approximation of the least squares problem $\tilde{\bf C} = \mbox{arg min}_{\mathbf{B} \in \mathbb{C}^{r \times r}} \| \Y_r - \X_r \mathbf{B} \|$ \cite{Drmac2017}. 
		Using the assumption that the matrix $\X_r$ has a full column rank,} the pseudoinverse is of the form $\X_r^{\dag} = (\X_r^{*} \X_r)^{-1} \X_r^{*}$ and the matrix $\tilde{\bf C}$ is 
	\begin{align}\label{eq:MPinverse}
	\tilde{\bf C} &= \X_r^{\dag}  \Y_r   = (\X_r^{*}\X_r )^{-1}  \X_r^{*} \Y_r   \nonumber \\
	&=   \left(\frac{1}{n} \X_r^{*} \X_r \right)^{-1} \left( \frac{1}{n}  \X_r^{*} \Y_r  \right) =  \tilde{\bf G}^{-1}  \left( \frac{1}{n}  \Y_r \X_r^{*}   \right) . 
	\end{align}
	{
		Here $\tilde{\bf G}$ denotes the numerical Grammian matrix whose elements are equal to
		\begin{align}\label{eq:Gij_element}
		\tilde{G}_{ij}(\omega,\x) & = {\textstyle\frac{1}{n}} < \f_{n}^{i-1}(\x), \f_{n}^{j-1}(\x)> \nonumber \\
		&=\frac{1}{n} \sum_{k=0}^{n-1} 
		\mathcal{K}^{i-1}f (T^{k} (\omega,\x)) \mathcal{K}^{j-1}f^{*} (T^{k} (\omega,\x)) \nonumber \\
		&= \frac{1}{n} \sum_{k=0}^{n-1} 
		(\mathcal{K}^{i-1}f)(\mathcal{K}^{j-1}f^{*}) ( T^{k} (\omega,\x)), \ i,j=1,\ldots,r.
		\end{align} 
		Now, by using (\ref{eq:Hankel_limit}), we conclude that
		\begin{equation}\label{eq:Gijlimit}
		\lim_{n \rightarrow \infty} \tilde{G}_{ij}(\omega,\x) = <  \mathcal{K}^{i-1} f,  \mathcal{K}^{j-1} f >_\mathcal{H}, \quad i,j = 1,\ldots,r \quad \text{for a.e. } \x
		\end{equation} 
		with respect to the measure $\mu$. From (\ref{eq:MPinverse}) we get that the last column $\tilde{\mathbf{c}} = \left( \tilde{c}_0(\omega,\x), \ \tilde{c}_1(\omega,\x), \ \ldots, \ \tilde{c}_{r-1}(\omega,\x) \right)^{T}$ of $\tilde{\bf C}$ is equal to 
		\begin{equation}\label{eq:ctilda}
		\tilde{\mathbf{c}} = \tilde{\bf G}^{-1}   {\textstyle\frac{1}{n}} 
		\left( < \f_{n}^{0}(\x),\f_{n}^{r}(\x)>, \ < \f_{n}^{1}(\x),\f_{n}^{r}(\x)>, \ \ldots, \ < \f_{n}^{r-1}(\x),\f_{n}^{r}(\x)> \right)^{T}.
		\end{equation}
		Since
		\begin{equation}\label{eq:cjlimit}
		\lim_{n \rightarrow \infty} < \f_{n}^{j-1}(\x), \f_{n}^{r}(\x)> = <  \mathcal{K}^{j-1} f, \mathcal{K}^{r} f >_\mathcal{H} \quad j = 1,\ldots,r,  \quad \text{for a.e. } \x
		\end{equation} 
		with respect to the measure $\mu$
		and the matrix $\tilde{\bf G}^{-1}$ converges to the inverse of the true Grammian matrix ${\bf G}^{-1}$, it follows that the components of $\tilde{\mathbf{c}}$ converge to the components of ${\bf c}$ (\ref{eq:c}). } As in {\cite[Proposition 3.1]{ArbabiMezic2016}}, we use now the fact that the eigenvalues of the matrix $\tilde{\bf C}$ are the roots of a polynomial having coefficient $\tilde{c}_i$. Therefore the eigenvalues of  $\tilde{\bf C}$ are continuously dependent on these coefficients, which implies that the eigenvalues of the companion matrix $\tilde{\bf C}$ converge to the eigenvalues of the exact companion matrix, i.e., to the exact eigenvalues of the stochastic Koopman operator.
	{ The eigenvalues and eigenvectors provided by DMD algorithm are computed from the eigenvalues and eigenvectors of the matrix ${\bf S}_r$ (\ref{eq:RayleighQ_S}), which is similar to the companion matrix $\tilde{\bf C}$ (see \cite[Proposition 3.1]{Drmac2017}), thus the conclusion follows.} 
\end{proof}

\color{black}
\section{Numerical examples}\label{sec:NumExamples}
{ In this section we illustrate the computation of the approximations of spectral objects of the stochastic Koopman operator on a variety of numerical examples. In all the examples we use the proposed methodologies in combination with the DMD RRR algorithm.} { First, we consider two discrete RDS examples: the Example \ref{nrot} and the linear RDS. Due to the fact that in the first case the semigroup property for the Koopman operators family is valid, the input matrices of DMD algorithm are defined by (\ref{eq:DMDRRR_XY_2}). We show that good approximations of the eigenvalues and eigenfunctions are obtained. In the second linear case, the DMD algorithm is applied on input matrices (\ref{eq:DMDRRR_XY_1})  and the differences between the exact and computed eigenvalues are determined. In Subsection \ref{sec:Numerics_ContLinearRDS} the approximation of time dependent principal eigenvalues for the continuous-time RDS generated by the linear RDE are determined by applying DMD RRR algorithm on data matrices defined by (\ref{eq:DMDRRR_XY_1}) where $\Y_m$ changes in time. In these two linear test examples, discrete and continuous one,  the solutions do not form Markov family, so that the associated stochastic Koopman operator family is not a semigroup. In Subsection \ref{sec:Numerics_SDE} we apply DMD RRR algorithm on input matrices defined by (\ref{eq:DMDRRR_XY_1}) and (\ref{eq:DMDRRR_XY_2}) to the RDS generated by scalar linear and scalar nonlinear SDE and approximate its eigenvalues and eigenfunctions. Finally, we apply the s-Hankel DMD algorithm to two RDS generated by SDE systems (stochastic Stuart-Landau and noisy Van der Pol system) whose solutions converge to the asymptotic limit cycles and to the RDS generated by noisy Lotka-Volterra system with an asymptotically stable equilibrium point. In most considered examples we first apply the DMD RRR algorithm to the related deterministic dynamical system (without noise) and then explore its behavior on the RDS.}

\subsection{Discrete RDS examples}\label{sec:Numerics_DiscreteRDS}

\subsubsection{Noisy rotation on the circle}
Consider the dynamical system defined in Example \ref{nrot}.  
We use the following set of observables: \[f_j(x) = \cos(j 2 \pi x), g_j(x) = \sin(j 2 \pi x),\ j=1,\ldots,n_1,\] arranged in the vector valued observable as \[\f=(f_1,\ldots,f_{n_1},g_1,\ldots,g_{n_1})^{T}.\]
{ The input matrices $\X_m$ and $\Y_m$ of the  DMD RRR algorithm are defined by (\ref{eq:DMDRRR_XY_2}). We consider both cases: deterministic and stochastic. In deterministic case, it follows from \cite{Tu2014jcd}, that for $m \gg n$, the matrix obtained by DMD algorithm is the approximation of the associated Koopman operator in a least square sense. Therefore, in this example, due to the ergodicity and as a consequence of the law of large numbers, as $m \rightarrow \infty$,  the obtained values provided by DMD algorithm converge to the eigenvalues and eigenvectors of the finite dimensional approximating matrix of the Koopman operator. According to \cite{TakeishiKawaharaYairi}, the algorithm converge in the stochastic case also to the eigenvalues and eigenfunctions of the stochastic Koopman operator. It is interesting to mention that to recover the spectrum of the stochastic Koopman operator in the considered ergodic system, it was enough to take in the matrices $\X_m$ and $\Y_m$ only the values of the observable vector on one long enough trajectory, i.e., to take $\X_m = (\f(\x_0) \ \f(T(\omega,\x_0)) \ldots  \f(T^{m-1}(\omega,\x_0)))$ and $\Y_m = (\f(T(\omega,\x_0)) \ \f(T^{2}(\omega,\x_0)) \ldots \f(T^{m}(\omega,\x_0)))$, where $\x_0$ denotes the chosen initial state.  }

We take $n_1=150$, which gives $n = 300$ observable functions, and use $m = 5000$ sequential snapshots to determine the matrices $\X_m$ and $\Y_m$. The numerical results for the flow with the  parameters $\vartheta = \pi/320$ and $\delta = 0.01$  obtained by using the DMD RRR algorithm are presented in Figure \ref{fig:rotation_detsto}. We note that the eigenvalues obtained in the deterministic case lie, as expected, on unit circle. The real parts of eigenfunctions $\phi_i(x)$, $i=1,2,3$, which we recover numerically, are presented in subfigure (c) of Figure \ref{fig:rotation_detsto}. They closely coincide with the theoretically established eigenfunctions given by (\ref{eq:SRotEF}). The results presented in subfigure (e) of Figure \ref{fig:rotation_detsto} show that numerically captured eigenvalues of the stochastic Koopman operator coincide very well with theoretical eigenvalues (\ref{eq:SRotEV}). In the stochastic case, the real parts of the eigenfunctions belonging to the first three eigenvalues $\lambda^{S}_i, i=1,2,3$ are presented in subfigure (f) of Figure \ref{fig:rotation_detsto}. As shown in Example 1, the eigenfunctions are given by (\ref{eq:SRotEF}). Two remarks are in order: 1) The deterministic and stochastic Koopman operators commute in this case, and thus the eigenfunctions are the same and the addition of noise does not perturb them (cf. \cite{Giannakis} where the numerically computed eigenfunctions in a similar commuting situation were not sensitive to the noise introduced in the system). 2) {The computed eigenvalues are the approximation of the eigenvalues associated with the  restriction of the  stochastic Koopman to the subspace spanned by the components of the vector observable. Since that subspace is invariant under the action of the Koopman operator, the eigenvalues are, as expected, computed with high accuracy.}

Results very similar to the presented ones were obtained when the observable vector is constructed with the observable functions: $f_j(x) = e^{j 2 \pi i x  }$ and $g_j(x) = e^{- j 2 \pi i x}$, $j=1,\ldots,n_1$.

\begin{figure}[!htb]
	\centering
	\includegraphics[width=1.\linewidth]{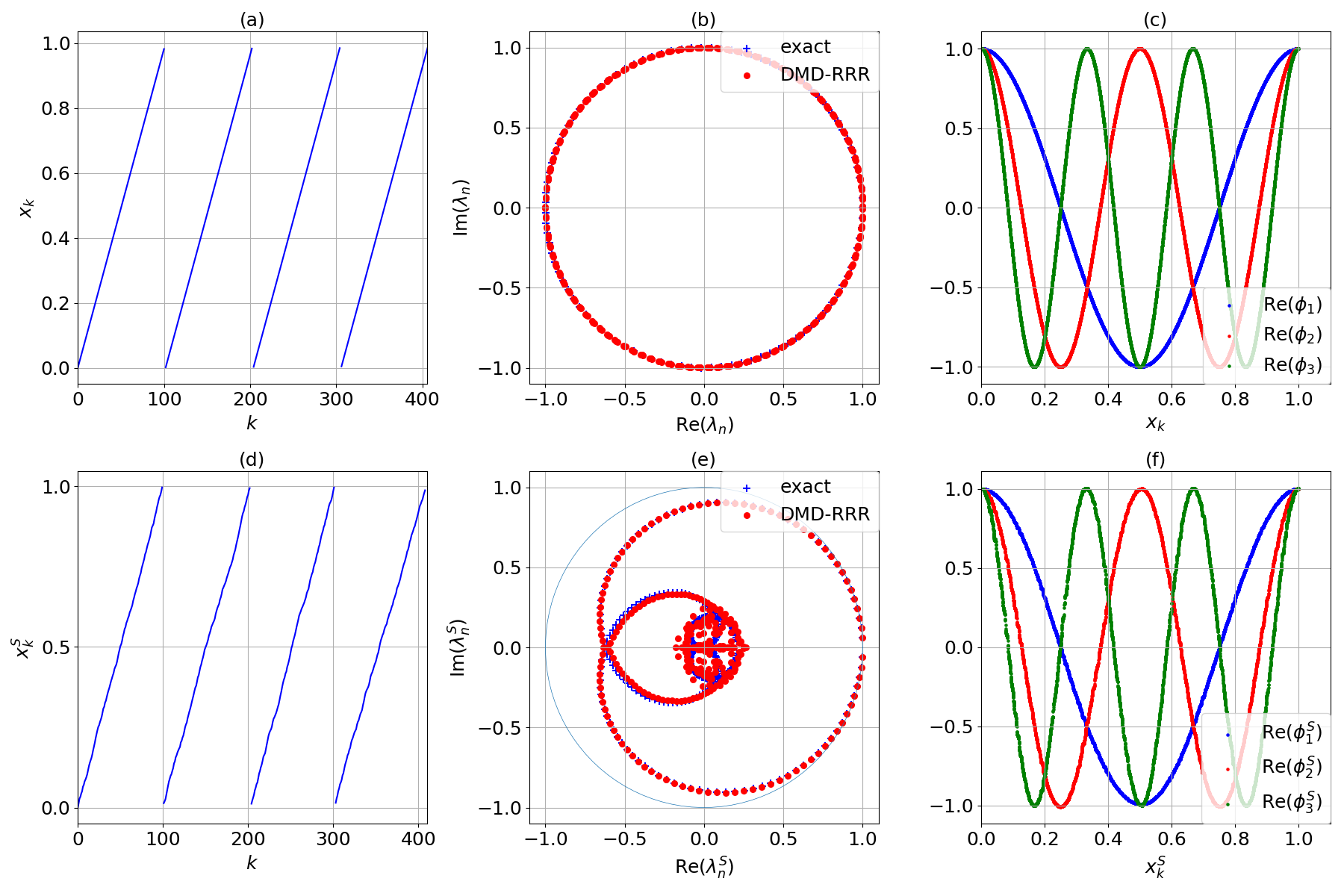}
	\caption{Rotation on circle defined by (\ref{eq:RotMap}) and (\ref{eq:RotStochMap}). Deterministic case $\theta=\pi/320$: (a) solution; (b) Koopman eigenvalues; (c) real part of eigenfunctions. Stochastic case $\theta=\pi/320$, $\delta=0.01$: (d) solution; (e) stochastic Koopman eigenvalues; (f) real part of eigenfunctions.}
	\label{fig:rotation_detsto}
\end{figure}

\subsubsection{Discrete linear RDS}
Here we consider the discrete linear RDS generated by a map (\ref{eq:LS_Discrete}) with 
\begin{equation}\label{eq:DTLRDSex_matrix}
\A(\omega) = 
\left(\begin{array}{cc}
0 & \pi(\omega) \\ 
-\pi(\omega) & 0
\end{array} \right),
\end{equation}
where $\pi$ is given by (\ref{eq:CanonProjection}). { We suppose that the dynamical system  $\theta=(\theta(t))_{t \in \mathbb{Z}^{+}}$ is defined by the shift transformations (\ref{eq:shift}).} 
The coordinates  $\omega_i$, $i\in \mathbb{Z}^{+}$ of $\omega\in\Omega$ are induced by i.i.d. random variables with the following distribution 
$$P(\omega_i = 1) = p_1, \ \ P(\omega_i=2) = 1-p_1, \qquad 0 < p_1 < 1.$$
As proved in Proposition \ref{LDRDSthm}, the principal eigenvalues of the stochastic Koopman operator are equal to the eigenvalues of the matrix $\hat{\A} = \mathbb{E}[\A(\omega)]$, while the associated eigenfunctions are given by (\ref{eq:LS_Discrete_ef}). We take $p_1=0.75$ and select $N = 10^{4}$ initial points uniformly distributed over $[0,1] \times [0,1]$. For every chosen initial point $\x_{j,0}, j=1,\ldots,N$, we determine the random trajectory $\x_{j,k}, k=1,2,\ldots$, where $\x_{j,k}$ denotes state value at $k$-th step.
The DMD RRR algorithm is applied to the full-state observables by taking states on each of the trajectory separately, i.e., we take as the algorithm input the matrices $\X_{m,j}, \Y_{m,j} \in \mathbb{R}^{2,m}$ defined by $\X_{m,j} = \left(\x_{j,0},\x_{j,1},\ldots,\x_{j,m-1}\right)$ and $\Y_{m,j} = \left(\x_{j,1},\x_{j,2},\ldots,\x_{j,m} \right)$, $j=1,\ldots,N$. In each computation { associated with initial value $\x_{j,0}$, we obtain the eigenvalue pair ${\lambda}^{(j)}_{1,2}$}, which should approximate the principal Koopman eigenvalues $\hat{\lambda}_{1,2}$. 
For fixed $m$ and for the chosen set of initial conditions, we get samples of approximating eigenvalues { ${\lambda}^{(j)}_{1,2}, j=1,\ldots,N$}. { To estimate the error,  for the obtained sets of approximating eigenvalues}, we evaluate the $L_1$, $L_2$ and $L_\infty$ norms of the difference between the exact eigenvalues and the computed eigenvalues. The norms determined for different values of parameter $m$ are presented in Figure \ref{fig:SSL_error}. One could notice that the accuracy of the obtained eigenvalues increases monotonically with the number of snapshots $m$ used in the computations, and that the error is  $\mathcal{O}(\frac{1}{\sqrt{m}})$, as would be expected for a random process of this type. { It is worth noticing that the norm of the state values $\|\x_{j,m}\|$ increases with $m$ and could become quite large for large values of $m$. Thus, the condition number of the input matrices of the DMD RRR algorithm could be very large. However, the scaling that is applied in this enhanced DMD RRR algorithm prevent the instabilities that otherwise could  be introduced if the standard DMD algorithm without scaling is applied.}

\begin{figure}[!htb]
	\centering
	\includegraphics[width=0.6\linewidth]{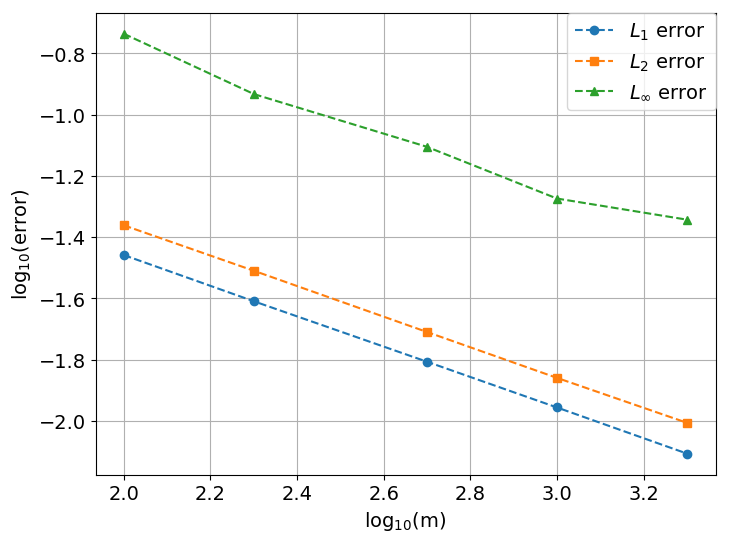}
	\caption{Discrete linear RDS defined by (\ref{eq:DTLRDSex_matrix}). $L_1$, $L_2$, and $L_\infty$ errors of approximated Koopman eigenvalues.}
	\label{fig:SSL_error}
\end{figure}

\subsection{Continuous-time linear RDS}\label{sec:Numerics_ContLinearRDS}
We consider now the continuous-time linear RDS that evolves according to (\ref{eq:LRDE}) where
\begin{equation}\label{eq:CTLRDSex_matrix}
\A(\omega) = 
\left(\begin{array}{cc}
\pi(\omega) & 1 \\ 
-b^2 & \pi(\omega)
\end{array} \right),
\end{equation}
{ $\theta(t)$ are shift transformations given by (\ref{eq:shift}), and $\pi$ is like in the previous test example given by (\ref{eq:CanonProjection}).
	It is supposed that the coordinates  $\omega(t)$, $t\in \mathbb{R}^{+}$ of $\omega\in\Omega$ are  piecewise constant of the form $\omega(t) = \sum_{i=1}^{\infty} \omega_i \mathbf{1}_{((i-1)\Delta t^{s}, i\Delta t^{s}]}(t)$, where $\Delta t^{s}$ is the chosen fixed time interval between switches and} $\omega_i$, $i \in \mathbb{Z}^{+}$ are i.i.d.~random variables with the following distribution
\begin{equation}\label{eq:wi-distribution}
P(\omega_i = a_1) = p_1, \ \ P(\omega_i=a_2) = 1-p_1, \qquad 0 < p_1 < 1.
\end{equation}
{ In the considered case, the obtained RDS is not Markovian and the stochastic Koopman operators  $(\mathcal{K}^{t})_{t \in \mathbb{T}}$ do not satisfy semigroup property.}  

{ We are interested in the numerical approximations of the eigenvalues of the associated stochastic Koopman operators. Therefore we apply the DMD RRR algorithm to the matrices $\X_m$ and $\Y_m$ defined using (\ref{eq:DMDRRR_XY_1}) and the full state observables. Using such approach, we expect to obtain the numerical approximations of the principal eigenvalues of the Koopman operators $\mathcal{K}^{t_k} = \mathcal{K}_{\Delta t}^{k}$, $k=0,1,2,\ldots$. Due to the switches of the matrices at time moments $i \Delta t_s$, we choose the time step $\Delta t$ so that $K \Delta t = \Delta t_s$ for some $K \in \mathbb{N}$.
	
	Since the matrices $A(\omega)$ commute and have simple eigenvalues, they are simultaneously   diagonalizable, so that according to Proposition \ref{thm2}, the stochastic Koopman eigenvalues of the operator $\mathcal{K}^{t}$ are equal $\lambda_j^{S}(t) = \mathbb{E}\left({\rm e}^{\int_{0}^{t} \lambda_j(\theta(s)\omega)  ds} \right)$, $j=1,2$. For large enough values of $t$, for example $t=N \Delta t^{s}$,
	after applying central limit theorem we get that the distribution of $\int_{0}^{t} \lambda_j(\theta(s)\omega)  ds = \sum_{i=1}^{N} \lambda_j(\omega_i) \Delta t^{s}$ is approximately normal with mean value $\hat{\lambda}_j(t) = \mathbb{E}(\lambda_j(\omega_i)) t$ and the variance $\hat{\sigma}_j^{2}(t) = V(\lambda_j(\omega_i)) \Delta t^{s} \,t$, where $V$ denotes variance.
	Here $\lambda_j(\omega_i), j=1,2$ are the eigenvalues of the matrix $\A(\omega_i)$, thus we get
	$$\hat{\lambda}_{1,2}(t)=(\hat{a}\pm bi)t, \quad \text{where    } \hat{a}={ p_1\, a_1+(1-p_1) a_2},$$
	and
	$$\hat{\sigma}_{1,2}^{2}(t) = p_1(1-p_1) (a_1-a_2)^{2} \Delta t^{s} t.$$
	It follows that ${\rm e}^{\int_{0}^{t} \lambda_j(\theta(s)\omega)  ds}$ is approximately log-normally  distributed with the mean value 
	\begin{equation}\label{eq:LognormalMean}
	\lambda_j^{S}(t) = \mathbb{E}\left({\rm e}^{\int_{0}^{t} \lambda_j(\theta(s)\omega)  ds} \right) = {\rm e}^{\hat{\lambda}_j(t) +\frac{1}{2}  \hat{\sigma}_{j}^{2}(t)}.
	\end{equation} }

{In the numerical computations, we take $a_1=-0.1$, $a_2=0.1$, $b=2$, and $\Delta t^{s}=\frac{\pi}{30}$.
	The numerical approximations of the expected values of the solution based on $N=100$ sampled trajectories and for the values of $p_1=0.75$, $p_1=0.5$, and $p_1=0.25$  are presented in Figure \ref{fig:linearrandomsolutions}. 
	By taking into account $m=100$ initial points, we approximate the eigenvalues of the Koopman operators $\mathcal{K}^{t_k}$, $t_k=k\Delta t$, where $\Delta t = \frac{\pi}{60}$. The relative error between the eigenvalues (\ref{eq:LognormalMean}) and the eigenvalues computed using DMD RRR algorithm is presented in Figure \ref{fig:linearRDS_computedeigenvalues}. One can conclude that the computations are stable since the relative errors in all three considered cases have similar magnitudes despite the fact that for $p_1=0.25$ the norm of the solution significantly increases. As in the previous example, this is a consequence of the scaling of data used in the DMD RRR algorithm. As expected, the variance of errors increase with time.     
}
\begin{figure}[tbh!]
	\centering
	\includegraphics[width=1.0\linewidth]{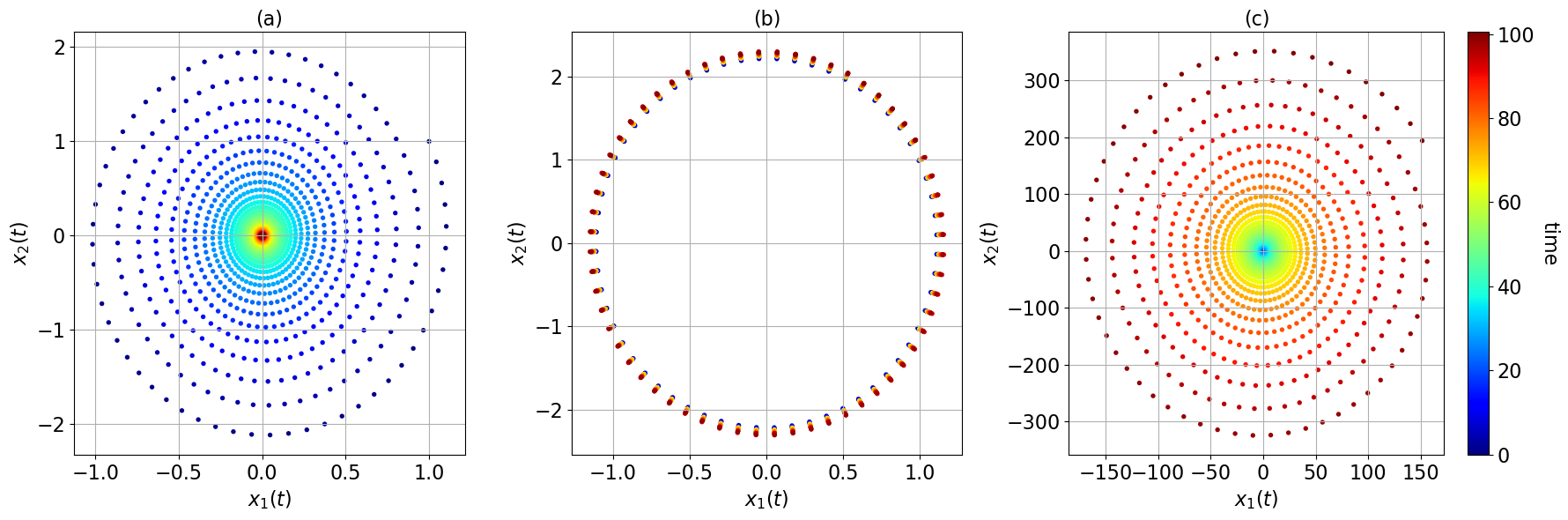}
	\caption{Numerical solutions $x^{S}_1(t)$, $x^{S}_2(t)$ of linear RDS defined by (\ref{eq:CTLRDSex_matrix}) for different values of parameter $p_1$: { (a) $p_1=0.75$ ; (b) $p_1=0.5$; (c) $p_1=0.25$}. The color coding refers to time $t$.}
	\label{fig:linearrandomsolutions}
\end{figure}

\begin{figure}[tbh!]
	\centering
	\includegraphics[width=1.0\linewidth]{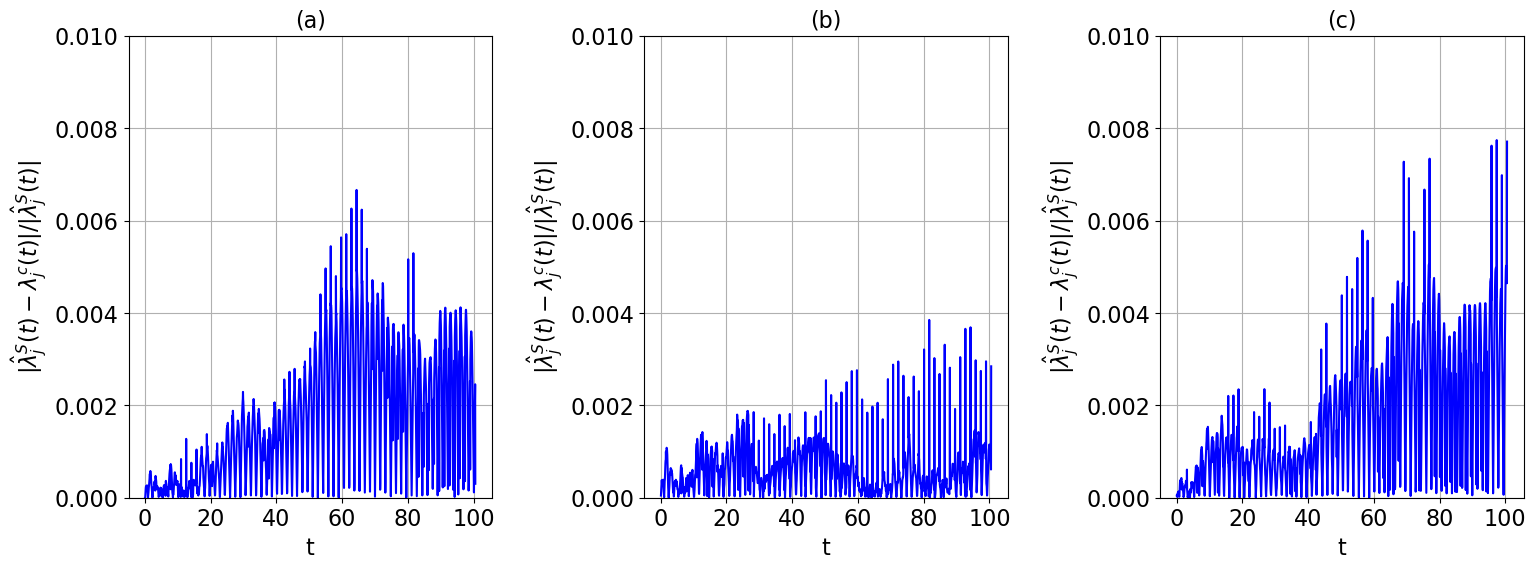}
	\caption{ Linear RDS defined by (\ref{eq:CTLRDSex_matrix}). The relative error between the eigenvalues $\hat{\lambda}_j^{S}(t)$ and the computed eigenvalues $\lambda^c_j(t)$ using DMD RRR algorithm for different values of parameter $p_1$: (a) $p_1=0.75$ ; (b) $p_1=0.5$; (c) $p_1=0.25$.}
	\label{fig:linearRDS_computedeigenvalues}
\end{figure}

\subsection{Stochastic differential equations examples}\label{sec:Numerics_SDE}
In this section we consider the flow induced by the autonomous SDE of the form (\ref{eq:SDE}) with random force accounting for the effect of the noise and disturbances of the system. In the considered examples the numerical solutions are determined by using the Euler-Mayurama method and the Runge-Kutta method for SDE developed in \cite{Rossler}. We use the fact that the convergence of Euler-Maruyama method is, for the examples we consider here, proved in \cite{HutzJent_SDE},
so that with very small time step we obtain the reference solutions and then check the numerical solutions obtained with Runge-Kutta method.
Nevertheless, the Runge--Kutta method performs significantly better  and needs less computational effort to attain the same accuracy since much larger time step can be used.

\subsubsection{Linear scalar SDE}
We consider the SDE
\begin{equation}\label{eq:LS_F}
dX = \mu X dt + \sigma dW_t,
\end{equation}
with $\mu < 0$ and $\sigma > 0$, { that generates the one-dimensional Ornstein-Uhlenbeck process \cite{Pavliotis}}. It is known that in the deterministic case, i.e., for $\sigma=0$, the Koopman eigenvalues are equal to
$$\lambda_n = n \mu,\quad n=0,1,2,\ldots$$
and the related Koopman eigenfunctions are 
$$\phi_n(x) = x^{n}.$$
In the stochastic case, the spectral properties of the stochastic Koopman generator  associated with the given equation were studied, for example in \cite{Gaspard_etal,Pavliotis} {(under the name backward Kolmogorov equation or Fokker-Planck equation)}.  Its eigenvalues are the same as in deterministic case, while the eigenfunctions are {(see \cite{Gaspard_etal,Pavliotis})}
$$\phi_n(x) = a_n H_n (\alpha x), \quad \textstyle \alpha=\sqrt{\frac{|\mu|}{\sigma}}, \quad n=0,1,2,\ldots.$$
Here $H_n$ denotes Hermite polynomials and $a_n$ are normalizing parameters.

It is known that the accuracy of the eigenvalues and eigenfunctions obtained by DMD algorithms is closely related to the choice of observable functions. In the deterministic case, when we take monomials $f_j(x) = x^{j}, j=1,\ldots,n$ as observable functions, satisfactory results for the first $n$ eigenvalues were obtained for moderate values of $n$, for example $n=10$,  while for larger values of $n$ significant numerical errors arise and only few eigenvalues were captured correctly. That is a consequence of highly ill-conditioned system when monomials of higher order $x^{n}$ are used, since their evolution, given by { $e^{n \mu t }x^{n}$, tends quickly to zero. Therefore the columns of the matrix $\X_m$ become highly linearly dependent.} In the case with $n=10$ and $m=2000$ snapshots, the ten leading eigenvalues are determined with the accuracy greater than $0.01$  (see Figure \ref{fig:ls_detsto}(b)). It is worth mentioning
that the standard DMD algorithm, which does not include scaling of data, was more sensitive and when it was applied on the same series of snapshots, at most three eigenvalues with the accuracy greater than $0.01$ were computed. 

In the stochastic case, the DMD RRR algorithm was applied on the same set of observable functions as in the deterministic case. { We use both approaches described in the Subsection \ref{subsec:DMDRRR} for defining input matrices $\X_m$ and $\Y_m$. In the first approach, for chosen $m=100$ initial points from the interval $[-1,1]$ and $n$ dimensional observable vector, we define $\X_m$, $\Y_m \in \mathbb{R}^{n\times m}$ by (\ref{eq:DMDRRR_XY_1}), where $\Y_m$ is determined for $k=100$ and by using $N=1000$ trajectories for each initial condition. In the second approach, for one chosen initial point, we generate $N=1000$ trajectories to determine the approximations of the expected values of the observable functions, which are then used in (\ref{eq:DMDRRR_XY_2}) to form input matrices $\X_m$ and $\Y_m$ for $m=2000$. The eigenvalues computed by using the first approach lead to very accurately computed eigenvalues of the Koopman operator (see Figure \ref{fig:ls_detsto}(d)) with the accuracy similar as in deterministic case. The accuracy depends on the number of computations $N$, as well as on the number and on the distribution of the initial points. In the second approach, the number of eigenvalues captured with satisfactory accuracy decreases and in the presented case only four of them are captured with the accuracy greater than $10^{-2}$ (see Figure \ref{fig:ls_detsto}(e)).  The eigenfunctions associated with this four eigenvalues are captured with satisfactory accuracy (Figure \ref{fig:ls_detsto}(f)). The reason for the lower accuracy could be a consequence of the errors that are introduced into the matrices $\X_m$ and $\Y_m$ when the expected values of the observable functions are approximated by averaging its values along the trajectories. When multiple initial conditions are used such type of error is introduced only in the matrix $\Y_m$ and not in the matrix $\X_m$.  The accuracy of the algorithm could be increased by increasing the number of computed trajectories since then, due to the law of large numbers, the more accurate approximations of expected values of the observables should be obtained. }  

\begin{figure}[!htb]
	\centering
	\includegraphics[width=1.0\linewidth]{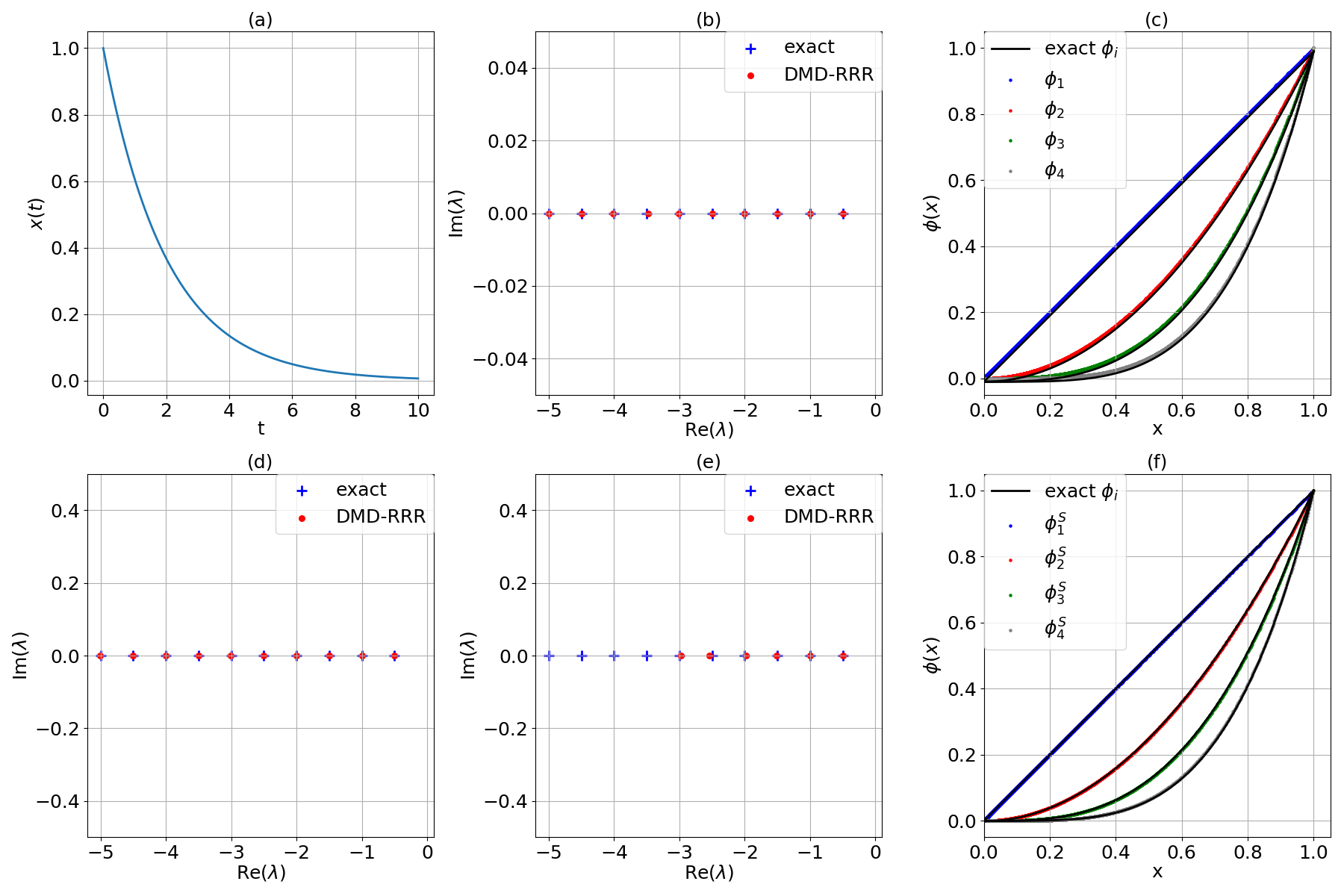}
	\caption{Linear scalar equation (\ref{eq:LS_F}). Deterministic case $\mu = -0.5$: (a) solution; (b) Koopman eigenvalues; (c) Koopman eigenfunctions;  Stochastic case $\mu = -0.5$  $\sigma = 0.001$: (d) stochastic Koopman eigenvalues - 1st approach: DMD RRR with (\ref{eq:DMDRRR_XY_1}); (e) stochastic Koopman eigenvalues - 2nd approach: DMD RRR with (\ref{eq:DMDRRR_XY_2}); (f) stochastic Koopman eigenfunctions.}
	\label{fig:ls_detsto}
\end{figure}

\subsubsection{{Nonlinear scalar SDE}}
One of the simplest nonlinear stochastic dynamical systems is the SDE of the form
\begin{equation}\label{eq:PB_F}
dX = (\mu X - X^{3}) dt + \sigma dW_t.
\end{equation}   
{  Its solutions are of different nature for $\mu>0$, $\mu=0$ and $\mu < 0$ \cite{Gaspard_etal} and therefore the related deterministic equation (for $\sigma=0$) usually appears in the literature when studying the bifurcation phenomena. We consider here the equation for fixed parameter $\mu < 0$. }
{ The eigenvalues and the eigenfunctions of the Liouville operator associated with this equation and its adjoint  were derived in \cite{Gaspard_etal}. Since (for $\sigma = 0$) the adjoint of the Liouville operator is actually the Koopman operator of the deterministic equation, we have from \cite{Gaspard_etal}} that the eigenvalues are equal to 
$$\lambda_n = n \mu, n = 0,1,2,\ldots,$$
while the associated eigenfunctions are
$$\phi_n(x)=\left(\frac{x}{\sqrt{x^{2}+|\mu|}}\right)^{n}.$$
In the stochastic case, when $\sigma>0$, the eigenvalues and the eigenfunctions of the generator of the stochastic Koopman operator family can be evaluated by solving the associated backward Kolmogorov equation, which can be transformed to the Schr\"{o}dinger equation (see \cite{Gaspard_etal}). We determine the approximations of eigenvalues numerically by solving Schr\"{o}dinger equation using finite difference method. For small values of $\sigma$ the eigenvalues are very similar to the deterministic case. 

As in the linear scalar case, we first apply the DMD RRR algorithm to compute the eigenvalues and the eigenfunctions of the Koopman operator in the deterministic case. The specific set of observables used in the computation is of importance for the accuracy of the results. When we use the analytically known eigenfunctions or their linear combinations as the observable functions, { the subspace spanned by these functions is closed under the action of the Koopman operator, so that its restriction on that subspace can be exactly described by a finite dimensional matrix, which is the key matrix used for obtaining outputs of DMD RRR algorithm. As a consequence, the eigenvalues and the eigenfunctions are computed with high accuracy. (see Figure \ref{fig:pb_detsto}).} On the other hand, when the set of observables was composed of monomials, the results were less accurate and only a few leading eigenvalues and eigenfunctions were computed with satisfactory accuracy. 

{ In the stochastic case, similarly as in the previous linear example, we use both DMD RRR approaches and the same set of observable functions as in deterministic case. The results obtained by using $m=100$ initial conditions, and the expected value of the observable functions determined over $N=1000$ trajectories and after $k=100$ time steps give the same number of accurately computed eigenvalues as in deterministic case (see Figure \ref{fig:pb_detsto}(d)).  On the other hand, when the second approach associated with one initial point is used, only first three eigenvalues are computed with satisfactory accuracy despite the fact that the same number of computed trajectories $N=1000$ and a large number of $m=2000$ snapshots was used. The reason for lower accuracy could be in the error introduced into the matrices $\X_m$ and $\Y_m$ as was described in the linear scalar case.}

\begin{figure}[!htb]
	\centering
	\includegraphics[width=1.0\linewidth]
	{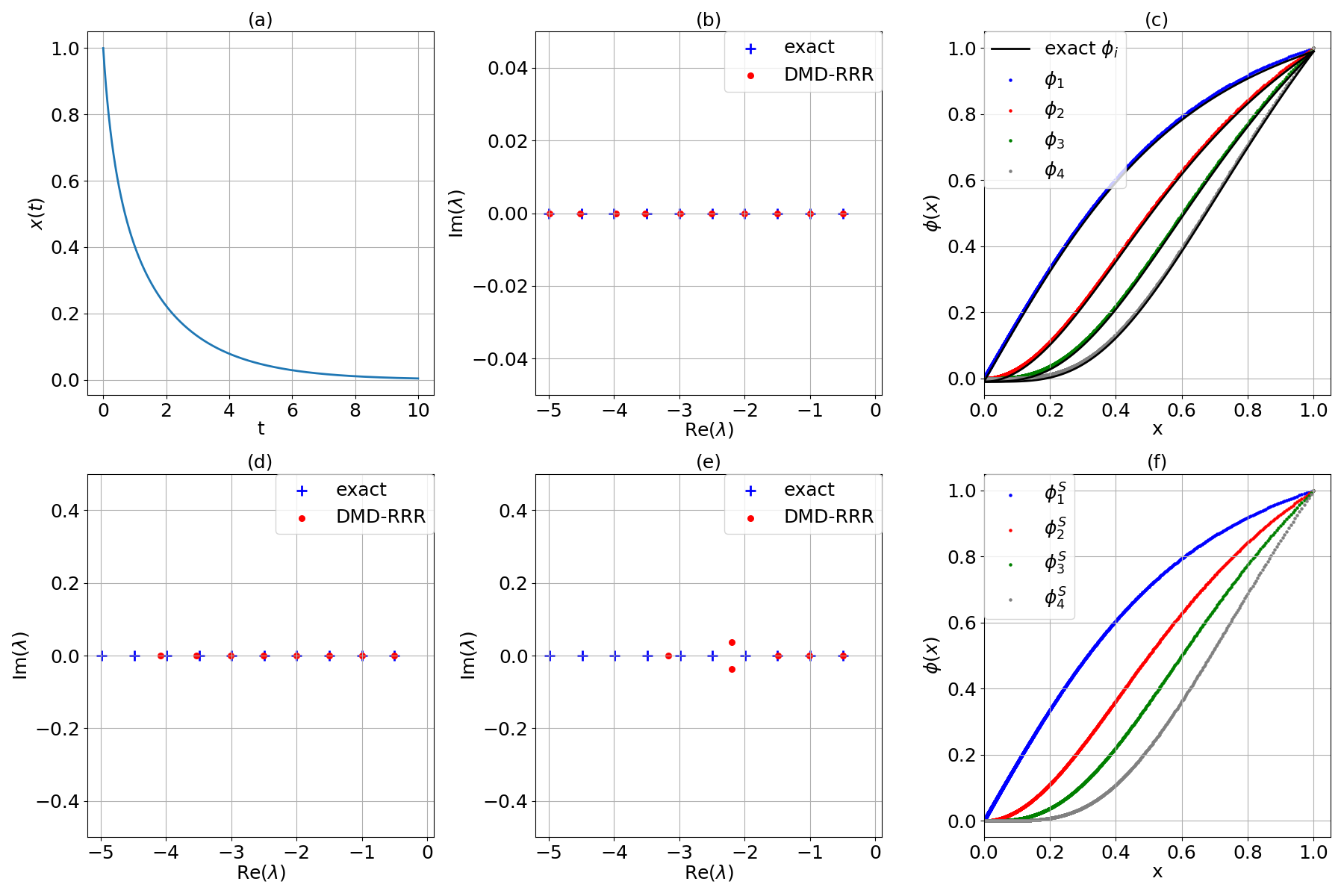}
	\caption{Nonlinear scalar equation (\ref{eq:PB_F}). Deterministic case $\mu = -0.5$: (a) solution; (b) Koopman eigenvalues; (c) Koopman eigenfunctions. Stochastic case $\mu = -0.5$  $\sigma = 0.001$: (d) stochastic Koopman eigenvalues - 1st approach: DMD RRR with (\ref{eq:DMDRRR_XY_1}); (e) stochastic Koopman eigenvalues - 2nd approach: DMD RRR with (\ref{eq:DMDRRR_XY_2}) ; (f) stochastic Koopman eigenfunctions.}
	\label{fig:pb_detsto}
\end{figure}

{
	\subsubsection{Stuart-Landau equations}
	The StuartÐLandau equations describe the behavior of a nonlinear oscillating system near the Hopf bifurcation. In polar coordinates $(r,\theta)$, the system reads
	\begin{align}\label{eq:SL_det}
	dr &= (\delta r - r^{3}) dt \nonumber \\
	d\theta &= (\gamma - \beta r^{2}) dt,
	\end{align}
	where $\beta, \gamma,$ and $\delta$ are parameters of the system. The behavior of the system depends on these parameters.
	The parameter $\delta$ associated with the Hopf bifurcation controls the stability of the fixed point $x^{*}$ $(\delta < 0)$ and of the limit cycle $\Gamma$ $(\delta > 0)$.
	We consider here the system with the limit cycle, i.e., for $\delta > 0$. According to \cite{TantetChekrounDijkstra}, the base frequency of the limit cycle $\Gamma: r=\sqrt{\delta}$ is equal to $\omega_0 = \gamma - \beta \delta$ and the eigenvalues of the system are
	$\lambda_{ln}=- 2l\delta + i n\omega_0, l \in \mathbb{N}, n \in \mathbb{Z}.$
	The stochastic Stuart-Landau equations are defined by
	\begin{align}\label{eq:SL_stoch}
	dr &= (\delta r - r^{3} + \frac{\epsilon^{2}}{r}) dt + \epsilon \,dW_r \nonumber \\
	d\theta &= (\gamma - \beta r^{2}) dt + \frac{\epsilon}{r} \, dW_\theta,
	\end{align}
	where $W_r$ and $W_\theta$ satisfy SDE system
	\begin{align*}
	dW_r &= \cos \theta dW_x + \sin \theta dW_y \nonumber \\
	dW_\theta &= - \sin \theta dW_x + \cos \theta dW_y,
	\end{align*}
	and $dW_x$ and $dW_y$ are independent Wiener processes. A detailed analysis of that system is provided in \cite{TantetChekrounDijkstra}.
	For small noise, controlled by the parameter $\epsilon$, the approximation of Koopman eigenvalues are derived in \cite{TantetChekrounDijkstra}, so that for the system with a stable limit cycle $\Gamma$ $(\delta > 0)$, the eigenvalues are given by
	\begin{equation}\label{eq:SL-eigenvalues}
	\lambda_{ln}=\begin{cases}
	- \frac{n^{2} \epsilon^{2}(1+\beta^{2})}{2 \delta} + i n \omega_0 + \mathcal{O}(\epsilon^{4}), l=0\\
	- 2l\delta + i n \omega_0 + \mathcal{O}(\epsilon^{2}), l>0
	\end{cases}.  
	\end{equation}
	
	In order to determine the Koopman eigenvalues, we apply the sHankel-DMD RRR algorithm to the scalar observable function of the form $f(r,\theta) = \sum_{k=1}^{K} {\rm e}^{\pm i k \theta}$, $k=1,\ldots,K$ in the deterministic case, and to the function $$f(r,\theta) = \sum_{k=1}^{K} {\rm e}^{\pm i k (\theta-\beta \log(r/\delta))}, k=1,\ldots,K$$
	in the stochastic case to calculate the eigenvalues $\lambda_{0,n}, n=1,\ldots,K$.
	The results for both cases, deterministic and stochastic, are presented in Figure \ref{fig:SL}. 
	For moderate values of $n$, quite good approximations of eigenvalues $\lambda_{0,n}$ are obtained. We mention that to obtain the base rotating frequency and corresponding multiples, the time span of snapshots used in the DMD RRR algorithm must include the basic period.  
	
	\begin{figure}[!htb]
		\centering
		\includegraphics[width=0.7\linewidth]
		{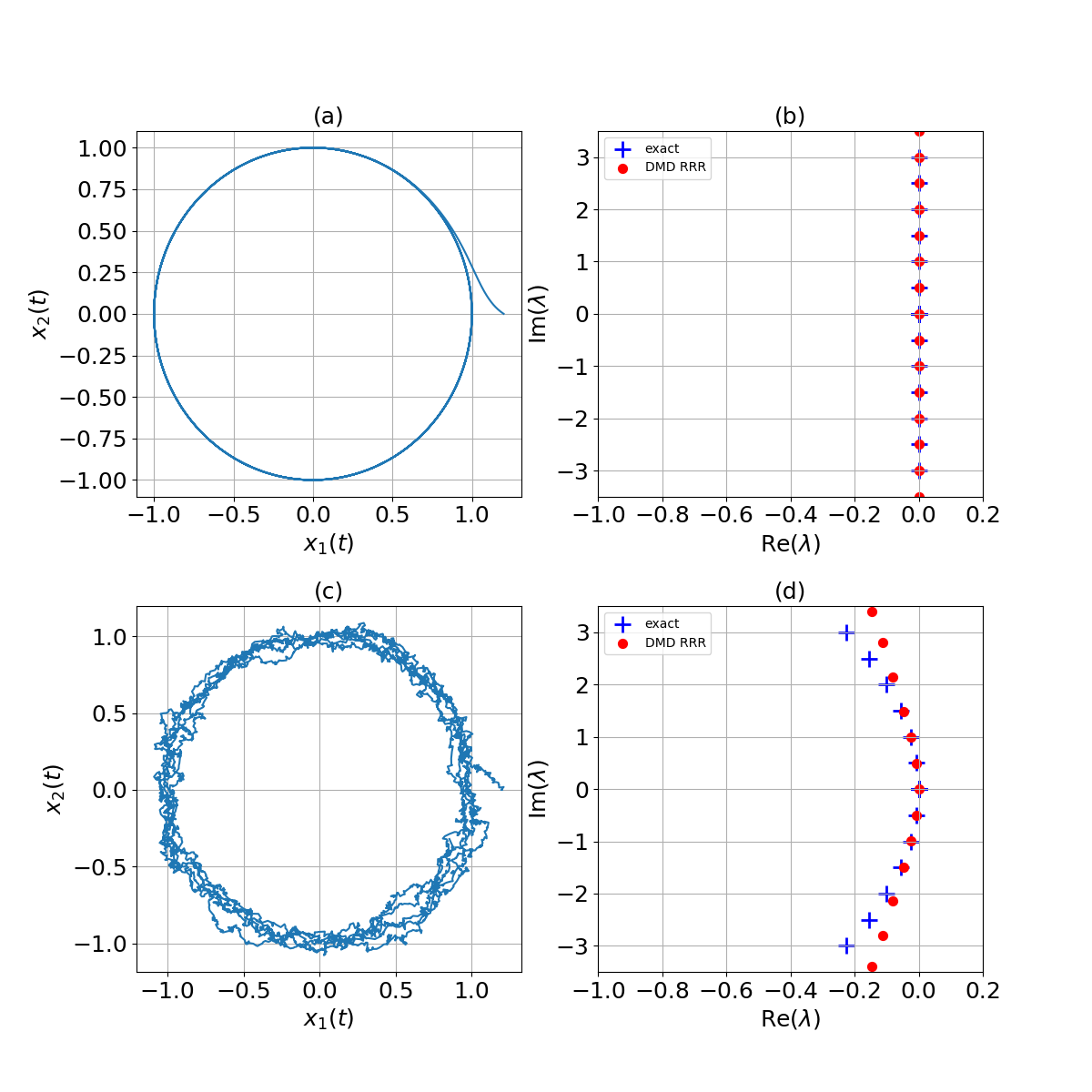}
		\caption{Stuart-Landau system with parameters $\delta=0.5, \beta=1,\gamma=1$. Deterministic case - equation (\ref{eq:SL_det}): (a) solution; (b) Koopman eigenvalues.  Stochastic case - equation (\ref{eq:SL_stoch}): (c) solution; (d) stochastic Koopman eigenvalues. The threshold for the residuals is set to $0.001$. }
		\label{fig:SL}
	\end{figure}
}

\subsubsection{Noisy van der Pol oscillator}
We analyze the two-dimensional test example in which the deterministic part is the system of two equations modeling standard Van der Pol oscillator. The stochastic part is modeled by the one-dimensional Wiener process so that the following system is considered
\begin{eqnarray}\label{eq:vanderPolOscillator}
dX_1 & = & X_2 dt \nonumber \\
dX_2 & = & \left( \mu (1-X_1^{2}) X_2 - X_1 \right) dt + \sqrt{2 \epsilon} dW_t.
\end{eqnarray}
It was proved in \cite{Arnold} that the solution of the considered SDE system exists. Furthermore, it was proved in \cite{HutzJent_SDE} that the numerical solutions computed with the Euler-Maruyama method converge to the exact solution of the considered SDE, which is also true for the Runge-Kutta method for SDE \cite{Rossler} used here. 

As in the previous examples, we consider first the deterministic system. It is well known that the dynamics converge to an asymptotic limit cycle whose basin of attraction is the whole $\mathbb{R}^{2}$. { According to \cite{Strogatz}, for the base frequency of the limit cycle the expression $\omega_0 = 1-\frac{1}{16}\mu^{2}+\mathcal{O}(\mu^{3})$ is valid,} thus for the parameter value $\mu = 0.3$ the basic frequency of the limit cycle is $\omega_0 \approx 0.995$. For the chosen initial state, the base frequency computed by applying the Hankel DMD RRR algorithm using the  scalar observable function $f(X_1,X_2)=X_1 + X_2 + \sqrt{X_1^{2}+X_2^{2}}$, is equal to $0.9944151$. 
In Figure \ref{fig:VPDMDRRR} we present the eigenvalues computed with the standard DMD algorithm and the eigenvalues obtained by using DMD RRR algorithm for which  the residuals of eigenvectors { determined using (\ref{eq:DMDresidual})} are smaller than the chosen threshold  { $\eta_0 = 10^{-3}$}. The number of eigenvalues and eigenvectors obtained by using both DMD algorithms is $250$ for the chosen Hankel matrix with $n=250$ columns and a larger number of rows, but for only few of them the residuals of the eigenvectors do not exceed the chosen threshold. The eigenvalues obtained using DMD RRR algorithm form a lattice structure containing the eigenvalues of the form $\{k \omega_0, - \mu + k \omega_0 : k \in \mathbb{Z} \}$, which is consistent with the theoretical results given in \cite{mezic:2016}. As evident from Figure \ref{fig:VPDMDRRR}, the standard DMD algorithm without residue evaluation computes many spurious eigenvalues.
In Figure \ref{fig:VPdetsto} we present the solution and the eigenvalues for which the residuals of eigenvectors do not exceed the chosen threshold in DMD RRR. The time evolution of the real part of the first three eigenfunctions of the Van der Pol system presented in the same figure coincide with the results presented in \cite{ArbabiMezic2016}.

The same approach as in the deterministic case is applied on the solution obtained from the stochastic system (\ref{eq:vanderPolOscillator}). From Figure \ref{fig:VPdetsto} we can conclude that some sort of stochastic asymptotic limit cycle appears, i.e., for large $t$ the solution is "smeared out" around the deterministic limit cycle. {Such statistical equilibrium can be determined by solving the associated Fokker-Planck equation. In \cite{FantuzzivanderPol} the energy bounds of the noisy van der Pol oscillator were determined. For small and moderate noise, the base frequency of the averaged limit cycle remains similar to the deterministic case \cite{LeungvanderPol}. Here we take the noise parameter $\epsilon=0.005$.
	The sHankel-DMD algorithm is applied to the same observable function as in the deterministic case and by using the same matrix dimension.  	} 
The eigenvalues and eigenfunctions of the related stochastic Koopman operator are not known analytically. { They can be approximated by the numerical discretization of the associated generator, however their evaluation is out of the scope of this paper. Under the assumption that the eigenvalues provided by the DMD RRR algorithm are good approximations for the exact eigenvalues, their negative real part suggests that the RDS in question is stable.} Again, as in the deterministic case, only the eigenvalues for which the residuals of the corresponding eigenvectors do not exceed the threshold $\eta_0 = 10^{-3}$ are selected and presented in Figure \ref{fig:VPdetsto}.
By comparing the results shown in Figure \ref{fig:VPdetsto}, we conclude that the numerically determined time evolution of eigenfunctions in the stochastic case is very similar to the evolution of eigenfunctions in the deterministic case.

\begin{figure}[!htb]
	\centering
	\includegraphics[width=0.7\linewidth]{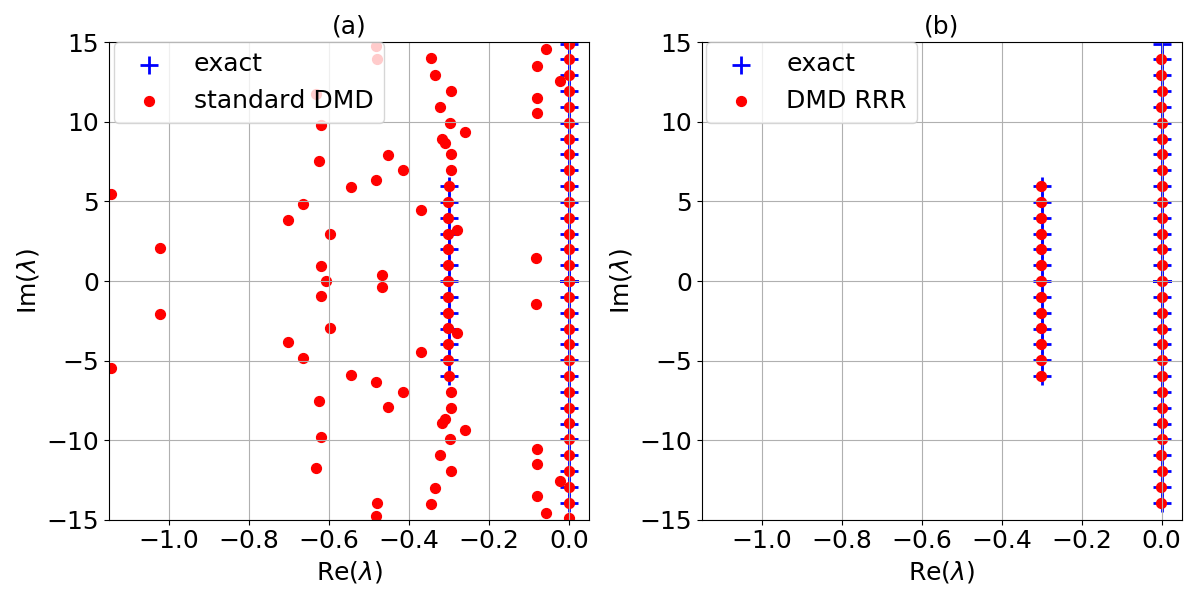}
	\caption{Van der Pol oscillator (\ref{eq:vanderPolOscillator}).  Deterministic case: (a) eigenvalues obtained by using standard DMD algorithm; (b) eigenvalues obtained by using DMD RRR algorithm with the threshold for the residuals equal to $0.001$.}	
	\label{fig:VPDMDRRR}
\end{figure}

\begin{figure}[!htb]
	\centering
	\includegraphics[width=1.0\linewidth]{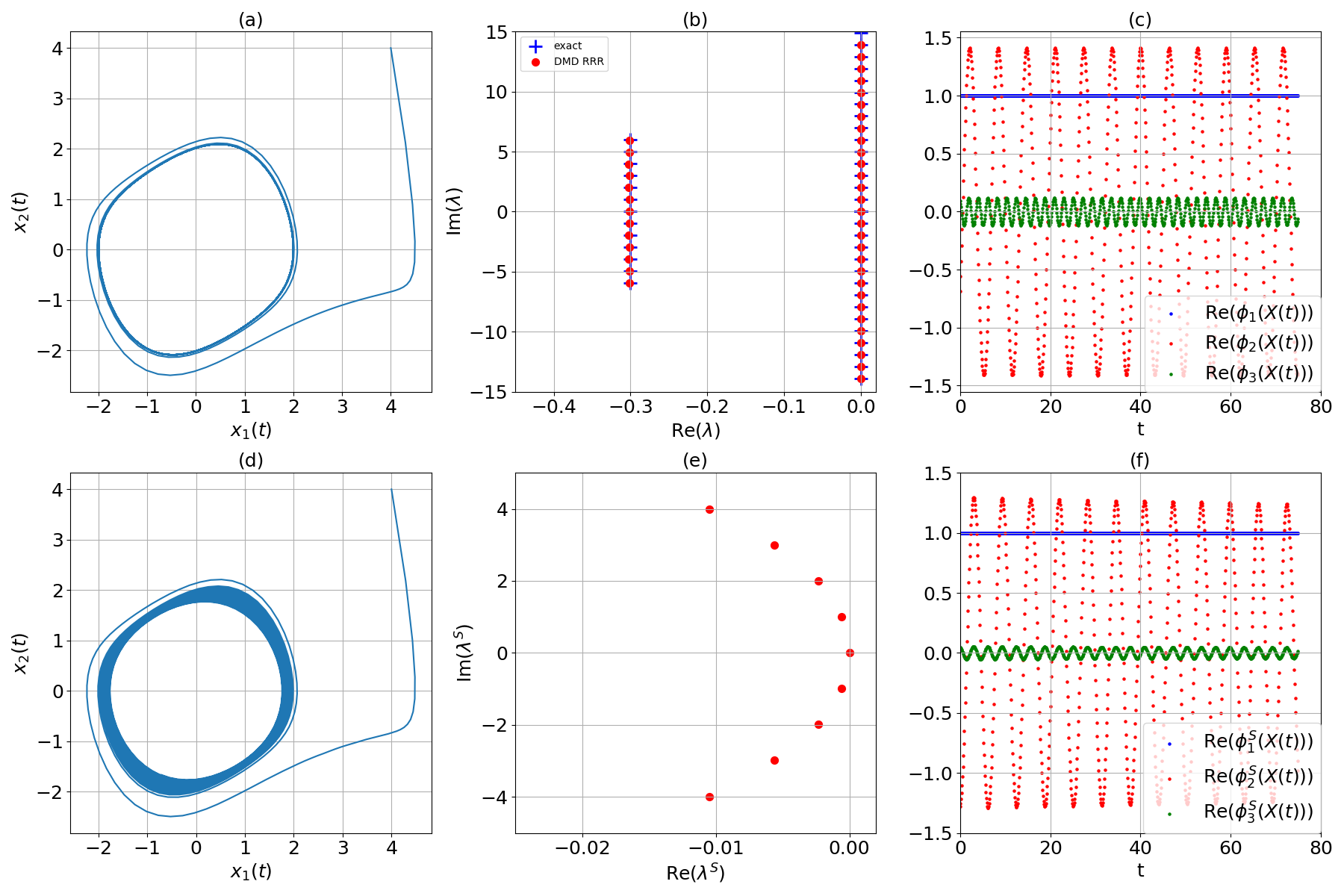}
	\caption{Van der Pol oscillator (\ref{eq:vanderPolOscillator}).  Deterministic case: (a) solution; (b) Koopman eigenvalues; (c) the time evolution of real part of Koopman eigenfunctions along trajectories. Stochastic case $\epsilon=0.005$: (d) solution; (e) stochastic Koopman eigenvalues; (f) the time evolution of real part of stochastic Koopman eigenfunctions along trajectories. The threshold for the residuals is set to $0.001$.}	
	\label{fig:VPdetsto}
\end{figure}

\subsubsection{Noisy Lotka-Volterra (predator-prey) system}
Here we consider the Lotka-Volterra system with competition, describing the predator prey system with two  species. Such type of a model is typically used by biologists and ecologists for modeling the population dynamics. The corresponding noisy system is given by: 
\begin{eqnarray}\label{eq:LotkaVoltera}
dX_1 & = & \left(a_1 - b_1 X_2 - c_1 X_1  \right) X_1  dt + \sigma_1 X_1 dW_t^{1} \nonumber \\
dX_2 & = & \left(- a_2  + b_2 X_1  - c_2 X_2 \right) X_2 dt + \sigma_2 X_2 dW_t^{2}.
\end{eqnarray}
The model parameters $a_1,b_1,c_1, a_2,b_2,c_2 > 0$ depend on the particular species under consideration. The intensity of the noise is modeled by nonnegative parameters $\sigma_1$ and $\sigma_2$. 

We first consider the properties of the deterministic system. The first quadrant $\mathbb{R}_+^{2}=\{(X_1,X_2) | \ X_1 \ge 0, X_2 \ge 0\}$ is the domain of evolution of the system. There is an equilibrium point on the $x_1$-axis $(\frac{a_1}{c_1},0)$. If the nullclines $a_1 - b_1 X_2 - c_1 X_1 = 0$ and $- a_2  + b_2 X_1  + c_2 X_2 = 0$ do not intersect, the equilibrium  point on the positive $x_1$ axis is unique. On the other hand, if these lines do intersect at the point $({x}^{*}_1,{x}^{*}_2)$, it can be shown that this is an asymptotically stable equilibrium point whose basin of attraction is $\mathbb{R}_+^{2} $ with the exception of the axes. The considered system can be linearized around the equilibrium point. As known from the Hartman-Grobman theorem, if all the eigenvalues of the Jacobian matrix at the fixed point have negative (or positive) real part, then there exist a $\mathcal{C}^{1}$-diffeomorphism ${\bf h}$ defining the conjugacy map that locally transform the system  to the linear one. Even more, as proved in \cite{LanMezic}, if all the eigenvalues have negative real part (i.e., if the fixed point is exponentially stable), the $\mathcal{C}^{1}$-diffeomorphism ${\bf h}$ is defined on the whole basin of attraction and the system is conjugate to the linear one globally on the basin of attraction. { This implies that the eigenvalues of the Koopman operator are the same as the eigenvalues of the Koopman operator associated with the linear system}. 

We select $a_1 = 1.0$, $b_1 = 0.5$, $c_1 = 0.01$, $a_2 = 0.75$, $b_2 = 0.25$, $c_2 = 0.01$. 
The off-axes equilibrium point is equal to $({x}^{*}_1,{x}^{*}_2) = (3.07754, 1.93845)$. Since the eigenvalues of the Jacobian matrix at this point are equal to $\lambda_{1,2} = -0.02500799 \pm 0.863524 \,i$, we conclude that this is an exponentially stable fixed point and that the system is conjugate to the linear one on its basin of attraction. 

To  numerically  evaluate the eigenvalues and eigenfunctions of the Koopman operator we use, as in the Van der Pol oscillator example, the Hankel DMD RRR algorithm, where the Hankel matrix is defined for the scalar observable function $f(X_1,X_2) = X_1 + X_2$. We use the matrix with $n=250$ rows and $m=100$ columns. By taking into account only the eigenvalues related to the eigenvectors with the residuals smaller than $\eta_0 = 10^{-3}$, several leading eigenvalues of the Koopman operator are captured with high accuracy (see Figure \ref{fig:PP_detsto}). The principal eigenvalues $\lambda_{1,2}$ are calculated with the accuracy greater than $10^{-6}$.

{ The solution and the asymptotic properties of noisy Lotka Volterra equation (\ref{eq:LotkaVoltera}) are considered in \cite{Arato}. In the stochastic case, for small enough parameters $\sigma_1$ and $\sigma_2$, one can determine the coordinates of the fixed point $(\bar{x}_1^{*},\bar{x}_2^{*})$ as the solution of equations: $a_1 - \frac{\sigma_1^{2}}{2} - b_1 X_2 - c_1 X_1 = 0$ and $a_2 - \frac{\sigma_2^{2}}{2} - b_1 X_2 - c_1 X_1 = 0$. According to \cite{Arato}, one can expect heuristically that $\xi_1 = X_1 - \bar{x}_1^{*}$ and $\xi_2 = X_2 - \bar{x}_2^{*}$ are components of the Ornstein-Uhlenbeck process in two dimensions, which is a solution of the corresponding linear SDE obtained by the linearization around the point $(\bar{x}_1^{*},\bar{x}_2^{*}) $. For small enough  $\sigma_1$ and $\sigma_2$, the fixed point $(\bar{x}_1^{*},\bar{x}_2^{*})$ remains asymptotically stable. 
	
	In the stochastic example that we consider here with the same parameters as in the deterministic case and for the variance parameters $\sigma_1 = \sigma_2 = 0.05$, the fixed point is equal to $(\bar{x}_1^{*},\bar{x}_2^{*})=(3.08243, 1.93585)$. After linearizing (\ref{eq:LotkaVoltera}) around this point, we get that the eigenvalues of the corresponding Jacobian matrix are equal to $\lambda_{1,2}^{S}=-0.02509 \pm 0.86363 i$. These eigenvalues coincide with the principal eigenvalues of the stochastic Koopman generator associated with the linear system.}

{
	Like in the deterministic case, we use the sHankel-DMD RRR algorithm, applied to the same observable function as before. The expectations are determined as the average values over $N=1000$ trajectories (see Figure \ref{fig:PP_detsto}). The principal eigenvalues of the Koopman operator in the considered stochastic case are evaluated by using the larger Hankel matrix than in deterministic case, with $n=750$ and $m=250$. Using the residual threshold $\eta_0 = 10^{-3}$, we chose only the eigenvalues related to the eigenvectors having small enough residual. The remaining set of eigenvalues is presented in Figure \ref{fig:PP_detsto}(d). One can see that the principal eigenvalues are determined quite accurately, however, the higher order ones are missing.}

\begin{figure}[!htb]
	\centering
	\includegraphics[width=0.7\linewidth]{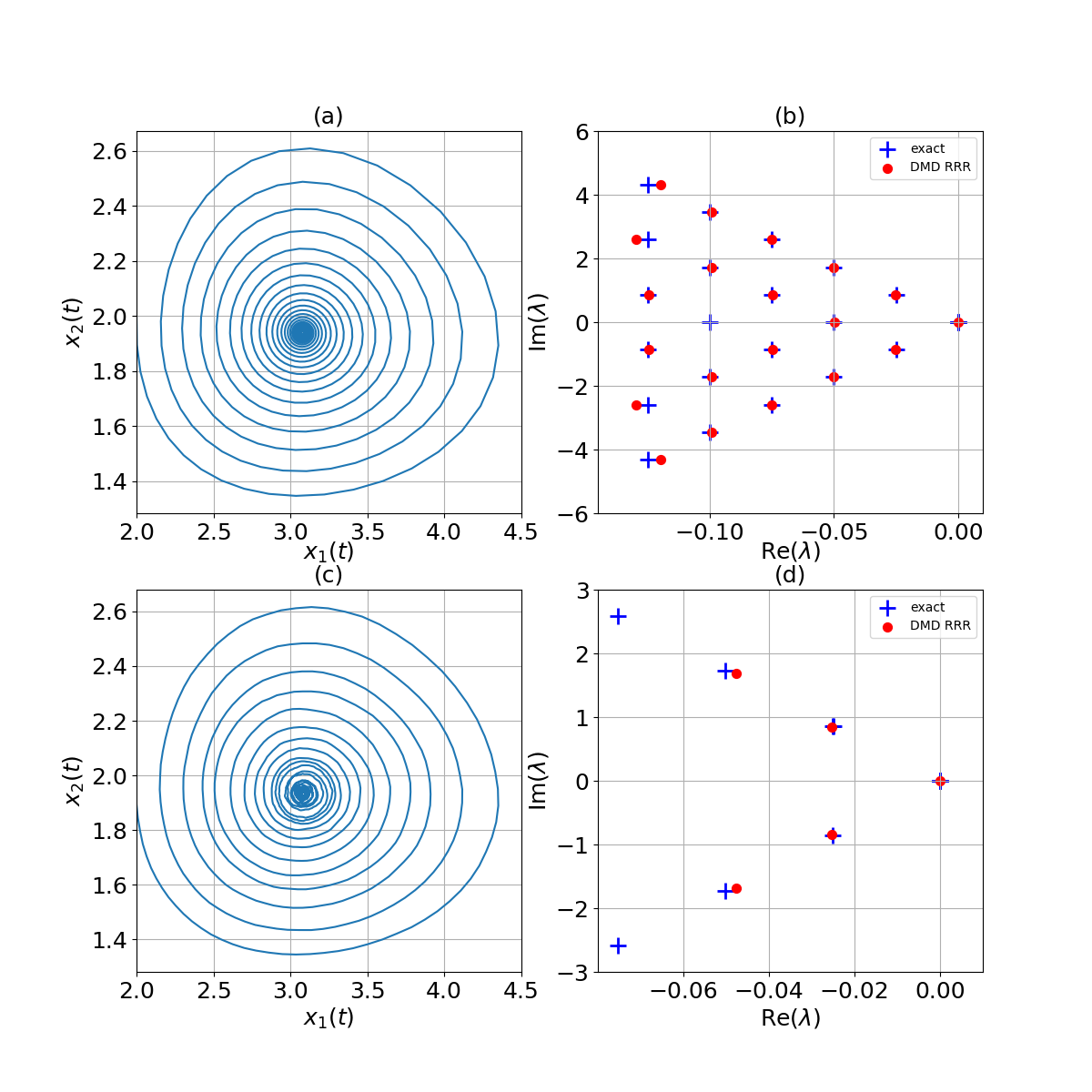}
	\caption{Lotka-Volterra system (\ref{eq:LotkaVoltera}). Deterministic case: (a) solution; (b) Koopman eigenvalues.  Stochastic case: (a) solution; (b) stochastic Koopman eigenvalues - the exact eigenvalues refer to the determined eigenvalues $\lambda_{1,2}^{S}$ that we heuristically expect to be valid in the stochastic case.
		The threshold for the residuals is set to $0.001$.}
	\label{fig:PP_detsto}
\end{figure}

{
	\section{Conclusions}
	In this work we consider the generators and spectral objects of stochastic Koopman operators associated with random dynamical systems. We explore  linear RDS, without invoking Markovian property, and (generally nonlinear) RDS with the Markov property that implies  the associated stochastic Koopman family is a semigroup. 
	
	To numerically compute the eigenvalues and the eigenfunctions of the stochastic Koopman operators family, we propose  approaches that enable use of  DMD algorithms in the stochastic framework for both Markovian 
	and non-Markovian cases. We define the stochastic Hankel DMD algorithm and prove that under the assumption of  ergodicity of RDS and the existence of finite dimensional invariant Krylov subspace, it converges to the true eigenvalues and eigenfunctions of the stochastic Koopman operator.
	
	The results presented for a variety of numerical examples show stability of the proposed numerical algorithms that reveal the spectral objects of the stochastic Koopman operators. Specifically, we use the DMD RRR algorithm \cite{Drmac2017} that enables precise determination of spectral objects via control of the residuals. 
	
	There are many interesting questions remaining in the study of stochastic Koopman operators. Some of the most intriguing ones relate to relationship of the geometry of the 
	level sets of stochastic Koopman operator eigenfunctions to the dynamics of the underlying RDS, initiated in \cite{mezicandbanaszuk:2004} for the case of asymptotic dynamics, and eigenvalues on the unit circle. We are currently pursuing such work for the more general class of eigenvalues off the unit circle and their associated eigenfunctions.
}

\section*{Acknowledgment}
This research has been supported by the DARPA Contract HR0011-16-C-0116 "On A
Data-Driven, Operator-Theoretic Framework for Space-Time Analysis of Process Dynamics" and AFOSR grants FA9550-08-1-0217 and FA9550-17-C-0012.
We are thankful to Milan Korda, Allan Avila for useful comments.

\bibliographystyle{plain}
\bibliography{bibliography}

\begin{thebibliography}{10}

\bibitem{Arato}
M.~Arat{\'o}.
\newblock A famous nonlinear stochastic equation: Lotka-volterra model with
  diffusion.
\newblock {\em Math. Comput. Modelling}, 38:709--726, 2003.

\bibitem{ArbabiMezic2016}
H.~{Arbabi} and I.~{Mezi{\'c}}.
\newblock Ergodic theory, dynamic mode decomposition and computation of
  spectral properties of the {K}oopman operator.
\newblock {\em SIAM J. Appl. Dyn. Syst.}, 16(4):2096--2126, 2017.

\bibitem{ArnoldSDE}
L.~Arnold.
\newblock {\em Stochastic Differential Equations}.
\newblock John Wiley \& Sons, 1974.

\bibitem{Arnold}
L.~Arnold.
\newblock {\em Random Dynamical Systems}.
\newblock Springer, 1998.

\bibitem{ArnoldCrauel}
L.~Arnold and H.~Crauel.
\newblock Random dynamical systems.
\newblock {\em Lypunov exponents. Proceedings Oberwolfach 1990. Springer
  Lecture notes in Mathematics.}, 1486:1--22, 1991.

\bibitem{bagheri:2013}
S.~Bagheri.
\newblock {K}oopman-mode decomposition of the cylinder wake.
\newblock {\em J. Fluid Mech.}, 726:596--623, 2013.

\bibitem{BudisicMezic_ApplKoop}
M.~Budi\v{s}i{\'c}, R.~Mohr, and I.~Mezi{\'c}.
\newblock Applied {K}oopmanism.
\newblock {\em Chaos}, 22(4):596--623, 2012.

\bibitem{CohenElliott}
S.~N. Cohen and R.~J. Elliott.
\newblock {\em Stochastic Calculus and Applications}.
\newblock Birkh\"auser, 2 edition, 2015.

\bibitem{Crauel_MarkovRDS}
H.~Crauel.
\newblock Markov measures for random dynamical systems.
\newblock {\em Stochastics and Stochastic Reports}, 37(3):153--173, 1991.

\bibitem{DaPratoZabczyk}
G.~Da~Prato and Zabczyk J.
\newblock {\em Stochastic equations in infinite dimensions}.
\newblock Cambridge University Press, 2 edition, 2014.

\bibitem{Drmac2017}
Z.~Drma\v{c}, I.~Mezi{\'c}, and R.~Mohr.
\newblock Data driven modal decompositions: analysis and enhancements.
\newblock {\em SIAM J. Sci. Comput.}, 40(4):A2253--A2285, 2018.

\bibitem{Dynkin}
E.~B. Dynkin.
\newblock {\em Markov processes, 1, 2, Grundlehren Math. Wiss., Vol.121, 122}.
\newblock Springer Verlag, 1965.

\bibitem{EngelNagel}
K.J. Engel and R.~Nagel.
\newblock {\em One-parameter semigroups for linear evolution operators}.
\newblock Springer, 2001.

\bibitem{FantuzzivanderPol}
G.~Fantuzzi, D.~Goluskin, D.~Huang, and S.I. Chernyshenko.
\newblock Bounds for deterministic and stochastic dynamical systems using
  sum-of-squares optimization.
\newblock {\em SIAM J. Appl. Dyn. Syst.}, 15(4), 2016.

\bibitem{Gaspard_etal}
P.~Gaspard, G.~Nicolis, A.~Provata, and S.~Tasaki.
\newblock Spectral signature of the pitchfork bifurcation: Liouville equation
  approach.
\newblock {\em Phys. Rev. E}, 51(1), 1995.

\bibitem{Giannakis}
D.~Giannakis.
\newblock Data-driven spectral decomposition and forecasting of ergodic
  dynamical systems.
\newblock {\em Appl. Comput. Harmon. Anal. in press}, 2018.

\bibitem{hemati2015biasing}
M.~S. Hemati, C.~W. Rowley, E.~A. Deem, and L.~N. Cattafesta.
\newblock De-biasing the dynamic mode decomposition for applied koopman
  spectral analysis.
\newblock {\em Theor. Comp. Fluid. Dyn.}, 31(4):349--368, 2017.

\bibitem{HornJohnson}
R.~A. Horn and C.~R. Johnson.
\newblock {\em Matrix Analysis}.
\newblock Cambridge University Press, 2 edition, 2013.

\bibitem{HutzJent_SDE}
M.~Hutzenthaler and A.~Jentzen.
\newblock Numerical approximations of stochastic differential equations with
  non-globally {L}ipschitz continuous coefficients.
\newblock {\em Mem. Amer. Math. Soc.}, 236(1112), 2015.

\bibitem{JungeMarsdenMezic}
O.~Junge, J.~E. Marsden, and I.~Mezi\'c.
\newblock Uncertainty in the dynamics of conservative maps.
\newblock In {\em 43rd IEEE Conference on Decision and Control}, 2004.

\bibitem{KlusKoltaiSchutte}
S.~Klus, P.~Koltai, and C.~Sch\"{u}tte.
\newblock On the numerical approximation of the {P}erron-{Frobenius} and
  {K}oopman operator.
\newblock {\em J. Comp. Dyn.}, 3(1):51--79, 2016.

\bibitem{KlusNuskeKoltaiSchutte}
S.~Klus, F.~N\"{u}ske, P.~Koltai, H.~Wu, I.~Kevrekidis, C.~Sch\"{u}tte, and
  F.~No\'{e}.
\newblock Data-driven model reduction and transfer operator approximation.
\newblock {\em J. Nonlinear Sci.}, 28(3):985--1010, 2018.

\bibitem{KlusSchutte}
S.~Klus and C.~Sch\"{u}tte.
\newblock Towards tensor-based methods for the numerical approximation of the
  {P}erron-{Frobenius} and {K}oopman operator.
\newblock {\em J. Comp. Dyn.}, 3(2):139--161, 2016.

\bibitem{Koopman:1931}
B.O. Koopman.
\newblock Hamiltonian systems and transformation in {H}ilbert space.
\newblock {\em P. Natl. Acad. Sci. USA}, 17(5):315, 1931.

\bibitem{KordaandMezic:2017}
M.~{Korda} and I.~{Mezi{\'c}}.
\newblock On convergence of extended dynamic mode decomposition to the
  {K}oopman operator.
\newblock {\em J. Nonlinear Sci.}, 28(2):687--710, 2018.

\bibitem{LanMezic}
Y.~Lan and I.~Mezi{\'c}.
\newblock Linearization in the large of nonlinear systems and {K}oopman
  operator spectrum.
\newblock {\em Physica D}, 242(1):42--53, 2013.

\bibitem{LasotaMackey}
A.~Lasota and M.~C. Mackey.
\newblock {\em Chaos, Fractals, and Noise, Stochastic Aspects of Dynamics}.
\newblock Springer, 2 edition, 1994.

\bibitem{LeungvanderPol}
H.K. Leung.
\newblock Stochastic transient of a noisy van der {P}ol oscillator.
\newblock {\em Physica A}, 221:340--347, 1995.

\bibitem{MacesicZicMezic}
S.~Ma{\'c}e{\v{s}}i{\'c}, N.~\v{C}rnjari{\'c} {\v{Z}}ic, and I.~Mezi{\'c}.
\newblock Koopman operator family spectrum for nonautonomous systems.
\newblock {\em SIAM J. Appl. Dyn. Syst.}, 17(4):2478--2515, 2018.

\bibitem{Mauroyetal:2013}
Alexandre Mauroy, Igor Mezi{\'c}, and Jeff Moehlis.
\newblock Isostables, isochrons, and koopman spectrum for the action--angle
  representation of stable fixed point dynamics.
\newblock {\em Physica D: Nonlinear Phenomena}, 261:19--30, 2013.

\bibitem{mezic:2005}
I.~Mezi{\'c}.
\newblock Spectral properties of dynamical systems, model reduction and
  decompositions.
\newblock {\em Nonlinear Dynam.}, 41(1):309--325, 2005.

\bibitem{Mezi2013}
I.~Mezi{\'c}.
\newblock Analysis of fluid flows via spectral properties of the {Koopman}
  operator.
\newblock {\em Annu. Rev. Fluid Mech.}, 45:357--378, 2013.

\bibitem{mezic:2016}
I.~{Mezi{\'c}}.
\newblock {Koopman Operator Spectrum and Data Analysis}.
\newblock {\em ArXiv e-prints}, February 2017.

\bibitem{mezicandbanaszuk:2004}
I.~Mezi{\'c} and A.~Banaszuk.
\newblock Comparison of systems with complex behavior.
\newblock {\em Physica D}, 197(1):101--133, 2004.

\bibitem{MezicandBanaszuk:2000}
I~Mezi{\'c} and Andrzej Banaszuk.
\newblock Comparison of systems with complex behavior: Spectral methods.
\newblock In {\em Decision and Control, 2000. Proceedings of the 39th IEEE
  Conference on}, volume~2, pages 1224--1231. IEEE, 2000.

\bibitem{suranamezic}
I.~Mezi\'{c} and A.~Surana.
\newblock Koopman mode decomposition for periodic/quasi-periodic time
  dependence.
\newblock {\em IFAC-PapersOnLine}, 49(18):690 -- 697, 2016.

\bibitem{Mezic:1994}
Igor Mezi{\'c}.
\newblock {\em On the geometrical and statistical properties of dynamical
  systems: theory and applications}.
\newblock PhD thesis, California Institute of Technology, 1994.

\bibitem{Mezic:2015}
Igor Mezi{\'c}.
\newblock On applications of the spectral theory of the {K}oopman operator in
  dynamical systems and control theory.
\newblock In {\em Decision and Control (CDC), 2015 IEEE 54th Annual Conference
  on}, pages 7034--7041. IEEE, 2015.

\bibitem{Mezic:2017}
Igor Mezi{\'c}.
\newblock Koopman operator spectrum and data analysis.
\newblock {\em arXiv preprint arXiv:1702.07597}, 2017.

\bibitem{MezicandWiggins:1999}
Igor Mezi{\'c} and Stephen Wiggins.
\newblock A method for visualization of invariant sets of dynamical systems
  based on the ergodic partition.
\newblock {\em Chaos: An Interdisciplinary Journal of Nonlinear Science},
  9(1):213--218, 1999.

\bibitem{Pavliotis}
G.~A. Pavliotis.
\newblock {\em Stochastic Processes and Applications: Diffusion Processes, the
  Fokker-Planck and Langevin Equations, Volume 60 of Texts in Applied
  Mathematics}.
\newblock Springer, 2014.

\bibitem{proctor}
J.L. Proctor and P.~A. Eckhoff.
\newblock Discovering dynamic patterns from infectious disease data using
  dynamic mode decomposition.
\newblock {\em Int. Health}, 7:139--145, 2015.

\bibitem{Rossler}
A.~R\"{o}\ss{}ler.
\newblock Runge--{K}utta methods for the strong approximation of solutions of
  stochastic differential equations.
\newblock {\em SIAM J. Numer. Anal.}, 48(3):922--952, 2010.

\bibitem{rowleyetal:2009}
C.W. Rowley, I.~Mezi{\'c}, S.~Bagheri, P.~Schlatter, and D.S. Henningson.
\newblock Spectral analysis of nonlinear flows.
\newblock {\em J. Fluid Mech.}, 641(1):115--127, 2009.

\bibitem{schmid:2008}
P.J. Schmid.
\newblock Dynamic mode decomposition of numerical and experimental data.
\newblock In {\em Sixty-First Annual Meeting of the APS Division of Fluid
  Dynamics}, 2008.

\bibitem{schmid:2010}
P.J. Schmid.
\newblock Dynamic mode decomposition of numerical and experimental data.
\newblock {\em J. Fluid Mech.}, 656(1):5--28, 2010.

\bibitem{schmid:2011}
P.J. Schmid, L.~Li, M.P. Juniper, and O.~Pust.
\newblock Applications of the dynamic mode decomposition.
\newblock {\em Theor. Comp. Fluid Dyn.}, 25(1):249--259, 2011.

\bibitem{sharma2016}
A.~S. Sharma, I.~Mezi{\'c}, and B.~J. McKeon.
\newblock Correspondence between koopman mode decomposition, resolvent mode
  decomposition, and invariant solutions of the navier-stokes equations.
\newblock {\em Phys. Rev. Fluids}, 1(3):032402, 2016.

\bibitem{ShnitzerTalmonSlotine}
T.~Shnitzer, R.~Talmon, and J.J Slotine.
\newblock Manifold learning with contracting observers for data-driven
  time-series analysis.
\newblock {\em IEEE T. Signal Proces.}, 65:904--918, 2017.

\bibitem{Strand}
J.L. Strand.
\newblock Random ordinary differential equations.
\newblock {\em J. Diff. Equations}, 7:538--553, 1970.

\bibitem{Strogatz}
S.~H. Strogatz.
\newblock {\em Nonlinear Dynamics and Chaos: with applications to physics,
  biology, chemistry and engineering}.
\newblock Perseus Books Publishing, 1994.

\bibitem{susukiandmezic:2012}
Y.~Susuki and I.~Mezi{\'c}.
\newblock Nonlinear {K}oopman modes and a precursor to power system swing
  instabilities.
\newblock {\em IEEE T. Power Syst.}, 27(3):1182--1191, 2012.

\bibitem{susukimezic}
Y.~Susuki and I.~Mezi\'c.
\newblock A {P}rony approximation of {K}oopman mode decomposition.
\newblock In {\em Decision and Control (CDC), 2015 IEEE 54th Annual Conference
  on Decision and Control}, 2015.

\bibitem{susukiandmezic:2016}
Y.~Susuki, I.~Mezi{\'c}, F.~Raak, and Hikihara T.
\newblock Applied {K}oopman operator theory for power system technology.
\newblock {\em Nonlinear Theory and Its Applications, IEICE}, 7(4):430--459,
  2016.

\bibitem{TakeishiKawaharaYairi}
N.~Takeishi, Y.~Kawahara, and T.~Yairi.
\newblock {Subspace dynamic mode decomposition for stochastic Koopman
  analysis}.
\newblock {\em Phys. Rev. E}, 96:033310, September 2017.

\bibitem{TantetChekrounDijkstra}
A.~{Tantet}, M.~D. {Chekroun}, H.~A. {Dijkstra}, and J.~D. {Neelin}.
\newblock {Mixing Spectrum in Reduced Phase Spaces of Stochastic Differential
  Equations. Part II: Stochastic Hopf Bifurcation}.
\newblock {\em ArXiv e-prints}, May 2017.

\bibitem{Tu2014jcd}
J.~H. Tu, C.~W. Rowley, D.~M. Luchtenburg, S.~L. Brunton, and J.~N. Kutz.
\newblock On dynamic mode decomposition: theory and applications.
\newblock {\em J. Comp. Dyn.}, 1(2):391--421, 2014.

\bibitem{williamsEDMD}
M.O. Williams, I.G. Kevrekidis, and C.W. Rowley.
\newblock A data-driven approximation of the {Koopman} operator: extending
  dynamic mode decomposition.
\newblock {\em J. Nonlinear Sci.}, 25(6):1307--1346, 2015.

\bibitem{Yosida}
K.~Yosida.
\newblock {\em Functional Analysis}.
\newblock Springer Verlag, 6 edition, 1980.

\end{thebibliography}

\end{document}